\documentclass{imanum}
\usepackage{graphicx}

\jno{drnxxx}
\received{}
\revised{}

\usepackage{amsmath,amssymb}
\usepackage{todonotes}
\usepackage{pgfplots, pgfplotstable}
\usepackage{tikz}
\usepackage{graphicx}
\usepackage{float}
\usepackage{subfig}
\usepackage{tikz}
\usetikzlibrary{patterns,calc,shapes,arrows,plotmarks,spy}

\usepackage{epstopdf}
\usepackage{algorithmic}
\usepackage{amsopn}

\newcommand{\R}{\mathbb{R}}
\newcommand{\N}{\mathbb{N}}

\newcommand{\vV}{\textbf{V}}
\newcommand{\vu}{\textbf{u}}
\newcommand{\vut}{u_\tau}
\newcommand{\vv}{\textbf{v}}
\newcommand{\vw}{\textbf{w}}
\newcommand{\vf}{\textbf{f}}
\newcommand{\vO}{\textbf{0}}
\newcommand{\vn}{\textbf{n}}

\newcommand{\pol}{\mathcal{P}}

\usepackage{mathrsfs}
\newcommand{\bpol}{\pol \!\!\!\!\!\!\!\!\!\!\pol \!\!\!\!\!\!\!\!\!\!\pol}
\newcommand{\bpolt}{\bpol_{\tau}}
\newcommand{\bpoln}{\bpol_{\vn}}

\newcommand{\cle}{ \preccurlyeq }
\newcommand{\cge}{ \succcurlyeq }

\newcommand{\laplace}{\Delta}
\newcommand{\Tc}{\mathcal{T}}
\newcommand{\Fc}{\mathcal{F}}
\newcommand{\jump}[1]{ {[\![} #1 {]\!]} }
\newcommand{\mean}[1]{ {\{\!\!\{} #1 {\}\!\!\}} }
\newcommand{\dgnorm}[1]{\| #1 \|_{1_h}}
\newcommand{\lnorm}[1]{\| #1 \|_{0}}
\newcommand{\lnormel}[1]{\| #1 \|_{0,T}}
\newcommand{\lnormbel}[1]{\| #1 \|_{0,E}}

\newcommand{\intd}{~\mathrm{d}}

\newcommand{\ext}{\mathcal{E}}
\newcommand{\exttang}{\mathcal{E}^{\tau}}
\newcommand{\extnorm}{\mathcal{E}^{\vn}}
\newcommand{\exttangc}{\mathcal{E}^{\tau,c}}
\newcommand{\Tref}{\widehat{T}}
\newcommand{\pop}{\mathcal{Z}}

\newcommand{\cov}{\mathcal{C}}
\newcommand{\intjacobi}{\widehat{p}}

\let\div\undefined\DeclareMathOperator{\div}{div}
\let\curl\undefined\DeclareMathOperator{\curl}{curl}

\newcommand{\projq}{\Pi^{Q_h}}
\newcommand{\fortin}{\Pi_{F}}

\begin{document}

\title{Polynomial robust stability analysis for $H(\textrm{div})$-conforming finite elements for the Stokes equations}
\shorttitle{Polynomial robust stability analysis}

\author{%
{\sc
Philip L. Lederer\thanks{Corresponding author. Email: philip.lederer@tuwien.ac.at}
and
Joachim Sch\"oberl} \\[2pt]
Institute for Analysis an Scientific Computing, TU Wien, Austria
}
\shortauthorlist{P. Lederer and J. Sch\"oberl}

\maketitle

\begin{abstract}
  {In this work we consider a discontinuous Galerkin method for the discretization of the Stokes problem. We use $H(\textrm{div})$-conforming finite elements as they provide major benefits such as exact mass conservation and pressure-independent error estimates. The main aspect of this work lies in the analysis of high order approximations. We show that the considered method is uniformly stable with respect to the polynomial order $k$ and provides optimal error estimates $\dgnorm{\vu - \vu_h} + \lnorm{\projq p-p_h} \le c \left( h/k \right)^s  \| \vu \|_{s+1} $. To derive those estimates, we prove a $k$-robust LBB condition. This proof is based on a polynomial $H^2$-stable extension operator. This extension operator itself is of interest for the numerical analysis of $C^0$-continuous discontinuous Galerkin methods for $4^{th}$ order problems.}
{Navier Stokes equations, mixed finite element methods, discontinuous Galerkin methods, high order methods}
\end{abstract}

\section{Introduction}
\label{sec:introduction}In this paper we consider the numerical solution of the Stokes equations on a bounded domain $\Omega \subset \R^2$, 
\begin{align} \label{int::stokesequations}
  \begin{array}{rcll}
    -\nu \Delta \vu + \nabla p &=& \vf &\textrm{in } \Omega \\
    \div \vu & = & 0 & \textrm{in } \Omega,                 
  \end{array}
\end{align}
with boundary conditions $\vu = \vu_D$ on $\partial \Omega$, where $\nu =const$ is the kinematic viscosity, $\vu$ is the velocity field, $p$ is the pressure and $\vf$ are external forces. The approximation of the (Navier-) Stokes problem is well analysed and many different finite element methods were introduced, see for example \cite{opac-b1081719, opac-b1119398, huerta, glowinski}. Furthermore discontinuous Galerkin (DG) finite element methods for elliptic problems got popular, see for example \cite{MR1885715, MR1897953, MR2431403},  and thus also for flow problems as in \cite{MR1938957,MR1974180, MR2085402,MR2136994, MR2031395, Cockburn2007,Cockburn:2002}. In this paper we consider an $H(\div)$-conforming method introduced in \cite{Cockburn2007} due to different advantages as local conservation, the possibility to use an upwinding scheme for convection dominated flows and pressure robust (independent) error estimates due to exact divergence free velocity test functions, see \cite{MR3133522, MR3326010, MR3460110, MR3481034}. To reduce the computational costs of DG methods, we also want to mention Hybrid DG (HDG) methods where new variables are introduced on the skeleton and a static condensation technique is used for the element unknowns, see \cite{MR2727822,MR2679797,MR2772094,MR2753354,MR2796169, MR3047948} and for $H(\div)$-conforming methods \cite{Lehrenfeld:thesis, guosheng,MR3511719}. \newline
The method we use is well analysed with respect to mesh refinement and provides optimal error estimates with respect to the mesh-size $h$. The main contribution of this paper is to show that the method is also uniformly stable with respect to the polynomial order $k$. For this we prove that the constant $\beta$ for the LBB condition 
\begin{align*}
  \sup_{\vO \neq \vv_h \in \vV_h} \frac{b(\vv_h,q_h)}{\| \vv_h \|_{\vV_h}} \geq \beta \| q_h \|_{Q_h} \quad \forall q_h \in Q_h,
\end{align*}
is independent of the order $k$. Together with standard continuity and ellipticity estimates this leads to a stable high order method. Note that with small adaptions of our results the polynomial robustness follows also for HDG methods as mentioned above.
High order methods for incompressible flow problems are of theoretical and practical importance. In \cite{opac-b1125338, Bernardi:1997:SMH} they consider a spectral method on the unit cube using polynomials of order $k$ and $k-2$ for the velocity and the pressure respectively. The resulting method leads to $\beta(k) = \mathcal{O}(k^{-\frac{d-1}{2}})$, where $d$ is the space dimension. The same method on triangles is discussed in \cite{opac-b1127030} with similar results. Furthermore the bad influence of a dependency of $k$ of the LBB constant for an iterative method for solving the Navier--Stokes equation was analysed. Understanding the problem, an improvement was achieved in \cite{MR1686546}. They used polynomials of partial order $k$ for the velocity and polynomials of total order at most $k-1$ for the pressure resulting in a uniformly stable method. Similar achievements for $hp$ mixed finite elements methods are accomplished in \cite{Stenberg1996}. Therein, different combinations of elements on quadrilaterals like continuous polynomials of order $k$ for the velocity and discontinuous polynomials of order $k-2$ for the pressure are discussed and an exact analysis is presented but again revealed a dependency on $k$. They also considered different tensor product spaces for each component of the velocity. A similar approach leading to an optimal exact divergence-free method was presented in \cite{MR2519595} using polynomials of order $k+1$ in the $x$ direction and polynomials of order $k$ in the $y$ direction for the velocity in $x$ direction and vice versa for the velocity in $y$ direction. Using proper degrees of freedom, this leads to a similar method on quadrilaterals as we use on triangles. Another approach, combining the tensor product structure on quadrilaterals and the advantage of approximating more complex geometries using triangles is analyzed in \cite{MR3518368}. The key of this method is to use the Duffy transformation and a proper pair of approximation spaces which leads to $\beta(k) = \mathcal{O}(k^{-\frac{1}{2}})$  with the drawback of using rational functions for the approximation.  We also want to mention the method considered in \cite{MR1897411} where a uniformly stable approximation using a continuous ansatz  for the velocity and pressure is presented which is adapted from the ideas of \cite{MR1686546} but enriches the pressure space to overcome the lack of convergence order that would appear using just a continuous version of this method. Considering continuous approximations, also the famous Taylor-Hood elements on triangles and quadrilaterals, see \cite{brezzifortin} and \cite{brezzi:falk} have to be discussed. Although these methods were shown to be stable with respect to the mesh size $h$, numerical evidence  predict that it is not uniformly stable with respect to $k$.
Of course high order methods were also used for discontinuous finite element methods. We want to mention the work of \cite{MR1938957} and \cite{MR1974180} where an analysis for $hp$-DG methods on quadrilaterals is presented but  revealed a dependency on the order $k$, and also the work of \cite{MR3047948} where an HDG method on triangles and quadrilaterals with similar results is introduced.
\newline \newline
The rest of the paper is structured in the following way. In chapter 2 we present the Stokes equation and the considered discretization method. Furthermore we present a short proof of the continuous divergence stability to motivate the existence of an $H^2$ stable polynomial extension operator which is used to prove the main theorem in chapter 3. In chapter 4 we take a look at some numerical examples and finally present the construction of an $H^2$ stable polynomial extension operator in chapter 5. 
\subsection{Preliminaries}
We assume an open bounded domain $\Omega \subset \R^2$ with a Lipschitz boundary $\Gamma$, thus for every point on the boundary there exists a Lipschitz-continuous mapping $\Phi^x$. If this mapping is furthermore differentiable up to order $m$ we say $\Gamma \in \mathcal{C}^{m,1}$. On $\Omega$ we define a shape regular triangulation $\mathcal{T}$  consisting of triangles $T$. Furthermore we assume $\mathcal{T}$ to be quasi--uniform thus, there exists one global mesh-size $h$ such that $h \approx \textrm{diam}(T),  \forall T \in \Tc$. The set of of edges, with respect to the triangulation $\mathcal{T}$, will be defined as $\mathcal{F}$. We call $\Tref := \{(x,y): 0\le x \le 1, 0 \le y \le 1, x+y \le 1 \}$ the reference element with the edges $E_1:= \{(x,0): 0\le x \le 1  \}$, $E_2:= \{(x,1-x): 0\le x \le 1  \}$ and $E_3:= \{(0,y): 0\le y \le 1  \}$, and define the interval $E:=\{(x,0): -1\le x \le 1  \}$. On all triangles we use $\vn$ and $\tau$ as symbols for the normal and tangential vector. For all subsets $\omega \subseteq\Omega$ with $\gamma := \partial \omega$ we have the space $L^2(\omega)$ with the norm $\| \cdot \|_{0,\omega}$ and the Sobolev spaces $H^1(\omega)$, $H^2(\omega), H^{1/2}(\gamma)$ with the corresponding Sobolev norms $ \| \cdot \|_{s,\omega}$. For a better readability we leave out the index $\omega$ if it is clear on which domain the norm is taken. 
On the edge $E_1$ we define the weighted $L^2$ and $H^{1/2}$ norm of a function $u$ as
\begin{align*}
\| u \|_{0^{*},E_1}^2 := \int_0^1 \left(\frac{1}{x} + \frac{1}{1-x}\right) u(x)^2 \intd s  \quad \textrm{and} \quad \| u \|_{1/2^*,E _1}^2 := | u |^2_{1/2,E _1} + \| u \|_{0^{*},E_1}^2.
\end{align*}
Furthermore we use the closed sub spaces
\begin{align*}
  L^2_0(\Omega) &:= \{ q \in L^2(\Omega): \int_\Omega q \intd x = 0\} \quad \textrm{and} \quad  H^1_0(\Omega) := \{ u \in H^1(\Omega): \mathrm{tr} ~ u = 0 ~\textrm{on}~ \partial \Omega\},
\end{align*}
and the polynomial spaces
\begin{align*}
  \pol^m(\Tc) &:= \{ v : v|_T \in \pol^m(T)~ \forall T \in \mathcal{T} \} = \prod_{ T \in \mathcal{T}} \pol^m(T) \quad \textrm{and} \quad   \bpol^m(\Tc) := [\pol^m(\Tc)]^2, \\
  \pol_{00}^m(E_1) &:= \{v \in \pol^m(E_1): v(0) = v'(0) = v(1) = v'(1) = 0\}.
\end{align*}
Also we define the following subspaces of vectorial polynomials on the reference triangle
\begin{align*}
 \bpol^m_{\tau}(\Tref):= \{\vv \in  \bpol^m(\Tref): \int_{\partial \Tref} \vv \cdot \tau = 0\} \quad \textrm{and} \quad \bpol^m_{\vn}(\Tref):= \{\vv \in  \bpol^m(\Tref): \int_{\partial \Tref} \vv \cdot \vn = 0\}.
\end{align*}
In this work we use an index notation for partial derivations, thus for an arbitrary function $u$ we write $u_{,x} := \frac{\partial u}{\partial x}$ and $u_{,y} := \frac{\partial u}{\partial y}$ and in a similar way second order derivations. We write $(x,y)^t$ as the transposed vector of $(x,y)$ and use $\perp$ as symbol for a counter clockwise  rotation by $\pi/2$, thus $(x,y)^\perp := (-y,x)$. Finally note that we use  $a \cle b$ when there exists a constant $c$ independent of $a, b$, the polynomial order and the mesh-size such that $a \le c b$.

\section{Discretization of the Stokes problem}
\label{sec:discretization}In this chapter we present the discretization of the stationary incompressible Stokes equations \ref{int::stokesequations} from \cite{Cockburn2007}, thus we use mixed order finite element spaces with the polynomial orders $k$ and $k-1$ for the velocity and the pressure respectively. To assure a local conservative and energy-stable method, we provide exact divergence-free velocity fields by using an $H(\div{})$-conforming method, thus every discrete velocity field $\vu_h$ is in
\begin{align*}
H(\div{})(\Omega) := \{ \vu \in [L^2(\Omega)]^2: \div \vu \in L^2(\Omega) \}.
\end{align*}
To ensure $\vu_h \in H(\div{})(\Omega)$ we demand normal continuity across each edge resulting in the approximation space for the velocity
\begin{align*}
  \vV_h := \{ \vu_h \in \bpol^k(T): \jump{\vu_h \cdot n} = 0 ~ \forall E \in \Fc \} \subset  H(\div{},\Omega),
\end{align*}
where $\jump{\cdot}$ is the usual jump operator.
For the pressure space we assume no continuity across edges 
\begin{align*}
  Q_h :=\prod\limits_{T \in \Tc}\pol^{k-1}(T) \cap L^2_0(\Omega).
\end{align*}
Note, that this pair of finite element spaces fulfills $\div \vV_h = Q_h$ and thus, a weakly incompressible velocity field $\vu_h \in \vV_h$ is also exact divergence free
\begin{align*}
\int_\Omega \div{\vu_h} ~ q \intd x= 0 \quad \forall q \in Q_h \quad \Rightarrow \quad  \div{\vu_h} = 0 \quad \textrm{in } \Omega.
\end{align*}
Furthermore, using $\mean{\cdot}$ as symbol for the mean value on an edge $E \in \Fc$  we define the bilinear-form 
\begin{align} \label{bilinearform::a}
  a(\vu_h,\vv_h) :=  \sum\limits_{T \in \Tc} \int_T \nu \nabla \vu_h : \nabla \vv_h \intd x &- \sum\limits_{E \in \Fc} \int_{E} \nu \mean{\nabla \vu_h \cdot n} \jump{\vv_h \cdot \tau} \intd s \\
  &- \sum\limits_{E \in \Fc} \int_{E} \nu \mean{\nabla \vv_h \cdot n}  \jump{\vu_h \cdot \tau} \intd s + \sum\limits_{E \in \Fc} \int_{E} \nu \frac{\alpha k^2}{h} \jump{\vu_h \cdot \tau}  \jump{\vv_h \cdot \tau} \intd s \nonumber
\end{align}
where $\alpha >0$ and $k^2/h$ is the stability coefficient similar to for example in \cite{MR3047948}, and the bilinear-form and linear-form 
\begin{align} \label{bilinearform::b}
b(\vu_h, q_h) := \sum\limits_{T \in \Tc} \int_T \div{\vu_h} ~ q_h \intd x \quad \textrm{and} \quad l(\vv_h):= \sum\limits_{T \in \Tc} \int_T \vf \cdot \vv_h \intd x.
\end{align}
The discrete Stokes problem now reads as: Find $(\vu_h, p_h)$ in $\vV_h\times Q_h$ such that
\begin{align} \label{dis::stokesproblem}
  \begin{array}{rcll}
      a(\vu_h,\vv_h) + b(\vv_h, p_h) &=& l(\vv_h)\quad &\forall \vv_h  \in  \vV_h \\
    b(\vu_h, q_h) & = & 0  \quad & \forall q_h  \in  Q_h.              
  \end{array}
\end{align}
On the spaces $Q_h$ and $\vV_h$ we use the $L^2$-norm $\lnorm{\cdot}$ and $\dgnorm{\vv_h}^2 := \sum\limits_{T \in \Tc} \lnormel{\nabla \vv_h}^2 + \sum\limits_{E \in \Fc} \frac{k^2}{h} \lnormbel{\jump{\vv_h \cdot \tau}}^2 $ respectively. 
\begin{lemma}\label{lemma:cont}For a proper choice of the stabilization parameter $\alpha >0$ in \eqref{bilinearform::a}, there exist constants $\alpha_1>0, \alpha_2>0$ and $\alpha_3>0$ independent of the mesh-size $h$ and the polynomial order $k$ such that $a(\cdot,\cdot)$ is coercive
  \begin{align*} 
    a(\vv_h,\vv_h) \ge \nu \alpha_3 \dgnorm{\vv_h}^2 \quad \forall \vv_h \in \vV_h,
  \end{align*}
  and $a(\cdot,\cdot)$ and $b(\cdot,\cdot)$ are continuous
  \begin{align*} 
    |a(\vu_h,\vv_h)| \le \nu \alpha_1 \dgnorm{\vu_h}\dgnorm{\vv_h} \quad \forall \vv_h, \vu_h \in \vV_h, \quad |b(\vu_h,q_h)| \le \alpha_2 \dgnorm{\vu_h} \lnorm{q_h} \quad \forall \vu_h \in \vV_h, \forall q_h \in  Q_h.
  \end{align*}

\end{lemma}
\begin{proof}
  The continuity properties of $a(\cdot,\cdot)$ and $b(\cdot,\cdot)$ follow by the definition of the norm $\dgnorm{\cdot}$. The coercivity follows with similar arguments as in \cite{MR2684358}.
\end{proof} 
\begin{theorem} \label{dis::maintheorem}
  There exists a constant $\beta > 0$ independent of the polynomial order $k$ and the mesh-size $h$ such that
  \begin{align*} 
\sup\limits_{\vO \neq \vv \in \vV_h} \frac{b(\vv_h, q_h)}{\dgnorm{\vv_h}} \ge \beta \lnorm{q_h} \quad \forall q_h \in Q_h.
  \end{align*}
\end{theorem}
\begin{proof}
The proof is presented in chapter \ref{sec:highorderinfsup}.
\end{proof}
\begin{theorem}
  Assume $(\vu_h, p_h)$ in $\vV_h\times Q_h$ is the solution of the discrete problem \eqref{dis::stokesproblem} and $(u,p)$ is the exact solution of the Stokes problem \eqref{int::stokesequations}. Furthermore assume regularity $\vu \in [H^{s+1}(\Omega)]^2$ and $p \in H^s(\Omega)$. Then there exists a constant $c_{err}$ independent of the mesh-size $h$ and the polynomial order $k$ such that  for $s \ge 1$ and $s \le k$ it holds
  \begin{align}\label{optimalerrorest}
\dgnorm{\vu - \vu_h} + \lnorm{\projq p-p_h} \le c_{err} \left( \frac{h}{k} \right)^{s}  \| \vu \|_{s+1},
  \end{align}
  where $\projq$ is the $L^2$ projector onto $Q_h$.
\end{theorem}
\begin{proof}
  In a first step we discretize the Poisson problem $-\nu \laplace \vu = f + \nabla p$ using a DG approximation, i.e. let $\vw_h$ be the solution of
  \begin{align*}
    a(\vw_h, \vv_h) = l(v_h) - b(\vv_h, p) \quad \forall \vv_h \in \vV_h,
  \end{align*}
  where we used that $\vv_h \in H(\div{})(\Omega)$, thus we used integration by parts for $\int_\Omega \nabla p \cdot \vv_h \intd x$. Using the estimate from chapter 3.2 in \cite{MR2684358} including the properties of the $L^2$ projector on triangles, equation (1.4) in \cite{MR3163888}, we get
  \begin{align} \label{errorestlaplace}
\dgnorm{\vu - \vw_h} \le \tilde{c}_{err} \left( \frac{h}{k} \right)^{s}  \| \vu \|_{s+1}.
  \end{align}
  Since $(\vu_h, p_h)$ is the solution of the discrete problem \eqref{dis::stokesproblem} we have
  \begin{align*}
    \begin{array}{rcll}
      a(\vu_h,\vv_h) + b(\vv_h, p_h) &=& b(\vv_h, p) + a(\vw_h, \vv_h)  \quad &\forall \vv_h  \in  \vV_h \\
    b(\vu_h, q_h) & = & 0  \quad & \forall q_h  \in  Q_h.              
    \end{array}
  \end{align*}
and thus
    \begin{align*}
    \begin{array}{rcll}
      a(\vu_h - \vw_h,\vv_h) + b(\vv_h, p_h-p) &=& 0  \quad &\forall \vv_h  \in  \vV_h \\
    b(\vu_h- \vw_h, q_h) & = & -b(\vw_h, q_h)  \quad & \forall q_h  \in  Q_h.              
    \end{array}
    \end{align*}
    Due to the property $\div \vV_h = Q_h$ we replace the term $b(\vv_h, p_h-p)$ in the first row by $b(\vv_h, p_h - \projq p)$. Using  $p_h - \projq p \in Q_h$ and the standard stability estimate of saddle point problems, see for example theorem 4.2.3 in \cite{brezzifortin}, we get
    \begin{align*}
      \dgnorm{\vu_h - \vw_h} + \frac{\beta}{\alpha_1} \| p_h - \projq p \|_0 &\le (1+ \frac{2\alpha_1}{\alpha_3\beta}) \| \div{\vw_h} \|_0 \\
      &=  (1+ \frac{2\alpha_1}{\alpha_3\beta}) \| \div{\vw_h} - \underbrace{\div{\vu}}_{=0} \|_0 \le (1+ \frac{2\alpha_1}{\alpha_3\beta}) \dgnorm{\vu - \vw_h}.
    \end{align*}
    The estimation \eqref{optimalerrorest} follows with \eqref{errorestlaplace}, the triangle inequality and the robustness of the constants $\alpha_1, \alpha_3$ and $\beta$ with respect to the mesh-size $h$ and $k$.
\end{proof}
\begin{remark} \label{remark::local}
  The introduced method can also be used for a non quasi--uniform triangulation, thus using triangles with different sizes $h_T$ and furthermore individual polynomial degrees $k_T$. By that we get a similar local error estimation as in \cite{MR2684358},
  \begin{align*}
\dgnorm{\vu - \vu_h} + \lnorm{\projq p-p_h} \le c_{err} \sqrt{ \sum\limits_{T \in \Tc}\left( \frac{h_T}{k_T} \right)^{2s}  \| \vu \|_{s+1,T}^2}.
  \end{align*}
\end{remark}
\subsection{Continuous LBB condition as motivation for an $H^2$ extension}
In this chapter we present a proof for the infinite dimensional version of the LBB condition of the Stokes problem as it can be found in \cite{Bernardi:1997:SMH}. For this we define the velocity space $\vV := [H^1_0(\Omega)]^2$,  the pressure space $Q := L^2_0(\Omega)$ and show that
\begin{align*}
      \sup\limits_{\vO \neq \vv \in \vV} \frac{b(\vv, q)}{\|\vv\|_1} \ge \beta_\infty \lnorm{q} \quad \forall q \in Q.
\end{align*}
Note that the LBB condition is equivalent to the existence of an $H^1$ stable right inverse of the divergence operator.
\begin{theorem}\label{dis::divisonto}
  Let $\Omega \subset \R^2$ be a bounded domain with a smooth Lipschitz boundary $\partial \Omega \in \mathcal{C}^{1,1}$. The divergence operator from $\vV$ to $Q$ is onto, so for every $q \in Q$ there exists a $\vv \in \vV$ such that
  \begin{align*}
    \div{\vv} = q, \quad \textrm{and} \quad \|\vv\|_1 \cle \lnorm{q}.
  \end{align*} 
\end{theorem}
\begin{proof}
  Let $q$ be an arbitrary function in $Q$. In the first step we consider the Poisson problem $\laplace \varphi = q$ in $\Omega$ with Neumann boundary conditions $\nabla \varphi \cdot \vn = 0$ on $\partial \Omega$.
  Due to the zero mean valaue of $q$ this problem has a unique solution in $H^1(\Omega) / \R$. Now set $\vv := \nabla \varphi$ to get $\div{\vv} = \laplace \varphi = q$ and using a regularity result for the Poisson problem we get  $\| v\|_1 = \| \varphi \|_2 \cle \lnorm{q}$. Furthermore note that we already have $\vv \cdot n = \nabla \varphi \cdot n = 0$ on the boundary $\partial \Omega$, so the idea is to construct a correction for the tangential component to satisfy the zero boundary values of $\vV$. Thus, we seek for a function $\psi \in H^2(\Omega)$ that fulfills 
  \begin{align*}
\psi = 0 \quad \textrm{on } \partial \Omega \quad \textrm{and} \quad     \frac{\partial \psi}{\partial n} =  - \vv \cdot \tau \quad \textrm{on } \partial \Omega \quad \textrm{and} \quad \| \psi \|_2 \cle \| \vv \|_1.
  \end{align*}
  The existence of such a function holds true since we have a smooth boundary (see Theorem 1.12 in \cite{Bernardi:1997:SMH}). Now we set $\tilde{\vv}:= \vv + \curl{\psi}$ to get $\div{\tilde{\vv}} = \div{\vv} + \div{\curl{\psi}} = q$ in $\Omega$ and on the boundary $\partial \Omega$ 
  \begin{align*}
    \tilde{\vv} \cdot \vn = \vv \cdot \vn + \curl{\psi} \cdot \vn = \nabla \psi \cdot \tau  = 0 \quad \textrm{and} \quad   \tilde{\vv} \cdot \tau = \vv \cdot \tau + \curl{\psi} \cdot \tau = \vv \cdot \tau +  \nabla \psi \cdot \vn  = 0.
  \end{align*}
Finally, due to the $H^2$-continuity of $\psi$, we get $\|\tilde{\vv}\|_1 = \|\vv\|_1 + \|\curl{\psi}\|_1 \cle \|\vv\|_1 \cle \lnorm{q}$.
\end{proof}
\\ \newline
It now follows immediately
\begin{align*} 
  \sup\limits_{\vO \neq \vv \in \vV} \frac{b(\vv, q)}{\|\vv\|_1} \cge \frac{\int_\Omega \div{\tilde{\vv}}{q} \intd x }{\|\tilde{\vv}\|_1} \cge \frac{\lnorm{q}^2}{\lnorm{q}} \cge \lnorm{q}.
\end{align*}
The crucial part of this proof was the existence of the correction $\psi$ which is stable in the $H^2$-norm. This holds true in the case of a smooth boundary of $\Omega$, but as the boundary of an element $T \in \Tc$ is just in $\mathcal{C}^{0,1}$ we get a problem when we want to adapt this proof to show the main theorem \ref{dis::maintheorem}. Such $H^2$-stable extensions of boundary values for non regular boundaries $\partial \Omega$ which are defined as a union of a finite number of regular parts are considered in \cite{grisvard1985elliptic} and \cite{bernardi1992}. For this they assume that {\it compatibility conditions} are satisfied at the points where two parts of the boundary gather. Those conditions are quite restrictive, thus do not hold true for all traces of polynomials, and that is why such a proof can not be used for example in the case of continuous velocity elements. Considering only normal continuous approximations with a tangential continuity only in a DG sense creates enough freedom to construct such an $H^2$-extension.


\section{Robust High order LBB estimation}
\label{sec:highorderinfsup}In this chapter we present the proof of theorem \ref{dis::maintheorem} in similar steps as in the proof of theorem \ref{dis::divisonto}. For this we assume the existence of a stable $H^2$-extension which is then presented in chapter \ref{sec:h2extension}.
\begin{theorem}[$H^2$-extension] \label{lbb:h2ext}
  For every $k$ there exists an operator $\ext : \bpoln^k(\Tref) \rightarrow \pol^{k+1}(\Tref)$ such that for $\vu_h  \in \bpoln^k(\Tref)$ it holds 
  \begin{align}
    &\curl{\ext(\vu_h)} \cdot \vn = \vu_h \cdot \vn \quad \textrm{on } \partial \Tref  \label{lbb:h2ext:propA} \\
    & \|(\vu_h - \curl{\ext(\vu_h)}) \cdot \tau \|_{0,\partial \Tref} \cle \frac{1}{k} \| \vu_h \|_{1,\Tref} \label{lbb:h2ext:propB} \\
    & \| \ext(\vu_h) \|_{2,\Tref} \cle \| \vu_h \|_{1,\Tref} \label{lbb:h2ext:propC}
  \end{align}
\end{theorem}
\begin{proof}
  See chapter \ref{sec:h2extension}.
\end{proof}
\newline \newline
We first show the LBB-condition on the reference triangle $\Tref$. For this we define the spaces $\widehat{\vV}_h := \bpol^k(\Tref)$ with $\widehat{\vV}_{h,0} := \{ \vv_h \in \widehat{\vV}_h: \vv_h \cdot n = 0 ~ \textrm{on } \partial \Tref\}$ and $\widehat{Q}_h :=  \pol^{k-1}(\Tref) \cap L^2_0(\Tref)$. The norm $\dgnorm{\cdot}$ on $\widehat{\vV}_h$ now reads as
\begin{align*}
\|\vv_h\|^2_{1_h,\Tref} = \|{\nabla \vv_h}\|^2_{0,\Tref} + \sum\limits_{E \in \partial \Tref} k^2 \| \vv_h \cdot \tau \|^2_{0,E} \quad \forall \vv_h \in \widehat{\vV}_h. 
\end{align*}
\begin{theorem}[local LBB-condition] \label{lbb::locallbbtheorem}
There exists a constant $\beta > 0$ independent of the polynomial order $k$ such that
  \begin{align*} 
\sup\limits_{\vO \neq \vv_h \in \widehat{\vV}_{h,0}} \frac{b(\vv_h, q_h)}{\|\vv_h\|_{1_h,\Tref}} \ge \beta \|q_h\|_{0,\Tref} \quad \forall q_h \in \widehat{Q}_h.
  \end{align*}
\end{theorem}
\begin{proof}
  Let $q_h \in \widehat{Q}_h$ be an arbitrary function. For a point $\tilde{x} \in \Tref$ we define a local Poincare operator $\pop_{\tilde{x}}:\pol^{k-1}(\Tref) \rightarrow [\pol^{k}(\Tref)]^2$ as introduced in \cite{MR2609313} by
  \begin{align*}
q(x) \mapsto \pop_{\tilde{x}}(q_h)(x):= (x-\tilde{x}) \int_0^1 t q_h(\gamma(t)) \intd t,
  \end{align*}
  with $\gamma(t) := \tilde{x} + t(x-\tilde{x})$, and by integrating over all points in $\Tref$ we furthermore define
  \begin{align*}
\vu_{h}^1(x) := \int_{\Tref} \theta(\tilde{x}) \pop_{\tilde{x}}(q)(x) \intd \tilde{x},
  \end{align*}
  where $\theta \in C^\infty_0(\Tref)$ is a smooth function. We observe that $\div{\vu_{h}^1} = q_h$ and $\| \vu_{h}^1\|_{1,\Tref} \cle \|q_h\|_{0,\Tref}$, see Corollary 3.4 in \cite{MR2609313}. As $q_h \in L^2_0(\Tref)$ we see that $\vu_{h}^1 \in \bpoln^k(\Tref)$ so we apply the extension operator of theorem \ref{lbb:h2ext} to define $\vu_{h}^2:= \curl{\ext(\vu_{h}^1)}$. By that we get for $\vu_h := \vu_{h}^1 - \vu_{h}^2$ and using property \eqref{lbb:h2ext:propA}
  \begin{align*}
    \div{\vu_h} = \div{\vu_{h}^1} - \div{\vu_{h}^2} = q_h  \textrm{ in } \Tref \quad \textrm{and} \quad \vu_h \cdot n = \vu_{h}^1 \cdot n - \vu_{h}^2 \cdot n =  0 \textrm{ on } \partial \Tref.
  \end{align*}
  Together with \eqref{lbb:h2ext:propC} and \eqref{lbb:h2ext:propB}
  we get
  \begin{align*}
    \|\vu_h\|^2_{1_h,\Tref} &=\|{\nabla \vu_{h}^1 - \nabla \vu_{h}^2}\|^2_{0,\Tref} +  \sum\limits_{E \in \partial \Tref} k^2  \| (\vu_{h}^1 - \vu_{h}^2) \cdot \tau \|^2_{0,E} \cle \|\nabla \vu_{h}^1\|^2_{0,\Tref} + \|{\nabla \vu_{h}^2}\|^2_{0,\Tref} +\| \vu_h \|_{1,\Tref}^2 \\ 
                      &\cle \| \vu_h \|_{1,\Tref}^2 + \|{\nabla \vu_{h}^2}\|^2_{0,\Tref}  = \| \vu_h \|_{1,\Tref}^2 + | \ext(\vu_{h}^1) |^2_{2,\Tref} \cle  \| \vu_h \|_{1,\Tref}^2  \cle  \|q_h\|_{0,\Tref}^2.
  \end{align*}
As $\vu_h \in \widehat{\vV}_{h,0}$ we bound the supremum from below thus
  \begin{align*}
  \sup\limits_{\vO \neq \vv_h \in \widehat{\vV}_{h,0}} \frac{b(\vv_h, q_h)}{\|\vv_h\|_{1_h,\Tref}} \cge \frac{\int_{\Tref} \div{\vu_h} ~{q_h} \intd x }{\|\vu_h\|_{1_h,\Tref}} \cge \frac{ \|q_h\|_{0,\Tref}^2}{ \|q_h\|_{0,\Tref}} \cge  \|q_h\|_{0,\Tref}.
  \end{align*}
\end{proof}
\newline
\begin{proof}[Proof of theorem \ref{dis::maintheorem}] For the proof we assume that $k \ge 2$. For the analysis of the lowest order case we refer to \cite{Lehrenfeld:thesis}. 
  We construct a Fortin operator $\fortin$ with 
  \begin{align}\label{fortin}
    b(\fortin \vu -\vu, q_h) = 0 \quad \forall q_h \in Q_h \quad \textrm{and} \quad \dgnorm{\fortin \vu } \le c_{F} \|\vu \|_1
  \end{align}
  where  $c_{F}$ is a robust constant with respect to $h$ and $k$. Then theorem \ref{dis::maintheorem} follows from  preposition 5.4.2 in \cite{brezzifortin}. We define the operator $\fortin$ as the sum of a low order operator $\fortin^1$ and a correction operator $\fortin^2$. The first operator is the standard operator for the  pair $\left( \bpol^2(\Tc) \cap [C^0(\Omega)]^2 \right) \times \pol^0(\Tc)$, see chapter 8.4 in \cite{brezzifortin}, which is uniformly continuous in $h$. For the second operator let $(\vw_h^T,r_h^T)$ be the solution of the local correction problem 
    \begin{align*}
    \begin{array}{rcll}
      a^T(\vw_h^T ,\vv_h) + b^T(\vv_h, r_h^T) &=& 0  \quad &\forall \vv_h  \in  \vV_h(T) \\
      b^T(\vw_h^T, q_h) & = & \int_T \div{(\vu - \fortin^1 \vu)} q_h \intd x  \quad & \forall q_h  \in  Q_h(T)
    \end{array}
    \end{align*}
    where $a^T(\cdot,\cdot)$ and $b^T(\cdot,\cdot)$ are the restrictions of the bilinear-forms \eqref{bilinearform::a} and \eqref{bilinearform::b}  on each triangle $T \in \Tc$ and $\vV_h(T)$ and $Q_h(T)$ are the corresponding local spaces. As $\div{(\vu - \fortin^1 \vu)} \perp \pol^0(\Tc)$ we get $\vw^T_h \cdot n = 0$ on $\partial T$ and together with the local LBB-condition theorem \ref{lbb::locallbbtheorem}, a scaling argument and the stability result of saddle point problems we furthermore get
    \begin{align*}
      \dgnorm{\vw^T_h} \cle \| \div{(\vu - \fortin^1 \vu)} \|_0 \le c_{2} \| \vu \|_1,
    \end{align*}
    where the constant $c_{2}$ is independent of $h$ and $k$. Now we set $\fortin^2 \vu := \sum\limits_{T \in \Tc} \vw^T_h$ and see that $\fortin = \fortin^1 + \fortin^2$ fullfills  \eqref{fortin} as for all $T \in \Tc$ it holds
    \begin{align*}
      \div{\fortin \vu}|_T &= \div{\fortin^1 \vu}|_T +  \div{\fortin^2 \vu}|_T = \div{\fortin^1 \vu}|_T +  \sum\limits_{T \in \Tc} \projq (\div{\vu}|_T - \div{\fortin^1 \vu}|_T) = \projq \div{\vu}|_T, \\
      \dgnorm{\fortin \vu } &\le \dgnorm{\fortin^1 \vu } + \dgnorm{\fortin^2 \vu } \cle \| \vu \|_1 + \sum\limits_{T \in \Tc}\dgnorm{\vw^T_h} \cle \| \vu \|_1.
    \end{align*}
\end{proof}
\begin{remark}
In the sense of remark \ref{remark::local} one can choose the polynomial order in theorem \ref{lbb:h2ext} as $k_T$, thus to show theorem \ref{dis::maintheorem} one can use the local LBB-condition theorem \ref{lbb::locallbbtheorem} with the proper order.
\end{remark}
\begin{remark}
As mentioned in remark 5.2.7 in \cite{brezzifortin} the existence of a Fortin operator $\fortin$ can be used to construct an error estimation independent of the LBB-constant $\beta_\infty$ of the infinite dimensional Stokes problem \eqref{int::stokesequations}. This may be an advantage for problems where $\beta_\infty$ is large.
\end{remark}

\section{Numerical examples}

Theorem \ref{dis::maintheorem} shows that the LBB-constant is independent of the polynomial order $k$. Beside the analysis in section~\ref{sec:highorderinfsup} numerical tests also show this independence. In Table \ref{lbb::numone} one sees the different values of $\beta$ on the reference triangle $\Tref$ for different polynomial orders $k$. Where the first line supports our results, the second line predicts that the polynomial robustness seems to hold also true in the case of quadrilaterals.\newline \newline
In the second example we take a closer look on the stability of the method introduced in section~\ref{sec:discretization}. For this we solve problem \eqref{dis::stokesproblem} on the unit square $\Omega=(0,1)^2$ where the exact solution is given by $\vu:=\curl(\sin(x)^2\sin(y)^2)$, $p=0$ and set $\nu=1$. For that we use a triangulation with $52$ elements with $h \approx 0.2$ and a stabilization parameter $\alpha = 4$.
In Table \ref{lbb::numtwo} we see the behavior of the error $\dgnorm{\vu-\vu_h}$, the error of the best approximation $\dgnorm{\vu-\vu_{best}}$ with $\vu_{best} := \arg \min\limits_{\vv_h \in \vV_h} \dgnorm{\vu-\vv_h}$ and the ratio $\dgnorm{\vu-\vu_h} / \dgnorm{\vu-\vu_{best}}$. The values support that the discrete Stokes solution is close to the best approximation as the ratio is practically close to one.
\begin{table}
  \centering
  \begin{tabular}{rcccc}
    \hline
    k & 4 & 8 & 16 & 32 \\
    \hline
    triangle & 0.167 & 0.190 & 0.201 & 0.205 \\
    quadrilateral & 0.305 & 0.313 & 0.315 & 0.315 \\   
    \hline
  \end{tabular}
  \caption{LBB constant in dependence of $k$ on the refernce triangle $\Tref$ \\and the reference quadrilateral $\widehat{Q}$.} \label{lbb::numone}
\end{table}

\begin{table}
  \centering
  \setlength{\tabcolsep}{7pt}
  \begin{tabular}{ccccccccc}
\hline
$k$&$2$&$4$&$6$&$8$\\
\hline
$\dgnorm{\vu-\vu_h}$&$1.587$&$3.343e-02$&$3.430e-04$&$2.258e-06$\\
\hline
$\dgnorm{\vu-\vu_{best}}$&$1.180$&$2.302e-02$&$2.309e-04$&$1.597e-06$\\
\hline
ratio&$1.344$&$1.452$&$1.486$&$1.414$\\
\hline
\end{tabular}
\caption{The error $\dgnorm{\vu-\vu_{best}}$, $\dgnorm{\vu-\vu_h}$ and the ratio for different $k$.} \label{lbb::numtwo}
\end{table}


\section{$H^2$ stable polynomial extension}
\label{sec:h2extension} In this chapter we prove the existence of a polynomial preserving $H^2$-stable extension operator. Note that in the two dimensional case the $\curl$ operator can be represented as a rotated gradient, thus on the boundary we have
\begin{align*}
\curl{\vu_h} \cdot \vn = \nabla^\perp \vu_h \cdot \vn = \nabla \vu_h \cdot \tau \quad \textrm{on } \partial T.
\end{align*}
For the ease of notation and readability, we switch to the gradient and change the tangential and normal vector in theorem \ref{lbb:h2ext}. Furthermore we skip the subscript $h$ of $\vu_h$, thus we write $\vu$ for a vectorial polynomial, and show instead:
\begin{theorem}[$H^2$-extension] \label{ext:h2ext}
   For every $k$ there exists an operator $\ext : \bpolt^k(\Tref) \rightarrow \pol^{k+1}(\Tref)$ such that for $\vu  \in \bpolt^k(\Tref)$ it holds 
  \begin{align}
    &\nabla {\ext(\vu)} \cdot \tau = \vu \cdot \tau \quad \textrm{on } \partial \Tref  \label{ext:h2ext:propA} \\
    & \|(\vu - \nabla \ext(\vu)) \cdot \vn \|_{0,\partial \Tref} \cle \frac{1}{k} \| u \|_{1, \Tref} \label{ext:h2ext:propB} \\
    & \| \ext(\vu) \|_{2,\Tref} \cle \| \vu \|_{1,\Tref} \label{ext:h2ext:propC}
  \end{align}
\end{theorem}
\begin{proof}
  The proof is provided in section~\ref{ext::proofmaintheorem}.
\end{proof}
\subsection{Literature and structure of this chapter}
Extension or lifting operators are a well discussed topic as they present the inverse map of trace operators which are known to be continuous, for example in the scalar case from $H^1(T)$ onto $H^{1/2}(\partial T)$. The challenge then is to construct an operator that maps functions from $H^{1/2}(\partial T)$ into $H^1(T)$. Furthermore a polynomial extension has the additional property that the operator maps a given polynomial on the boundary onto a polynomial on the element. The importance of such extensions have their origin mostly in the analysis of $p$- and $hp$- finite element methods, see for example in \cite{MR2034876}, spectral methods and preconditioning as for example in the work of \cite{MR2387903}. Although the existence  of polynomial $H^1$-, $H(\div{})$- and $H(\curl{})$-extensions are already well analysed, to the best of our knowledge a stable polynomial $H^2$-extension, which is presented in this chapter, is the first result of this kind. Beside the application of the $H^2$-stable extension operator in the proof of the local LBB condition theorem \ref{lbb::locallbbtheorem}, this operator may be of interest for example to construct proper interpolation operators for $4^{th}$ order $C^0$-continuous DG methods, see for example the works of \cite{MR2670114, MR3022211}. \newline Before we present our results we want to mention some literature as many techniques we use are motivated by their accomplishments. The first work we consider is the pioneering paper of \cite{Babuska1987} which contains major ideas that are developed and adapted in later contributions including the technique of splitting the operator in primary and proper correction liftings. For polynomial preserving extensions in two dimensions we want to mention the work of \cite{MR1055459,MR1076961} and for three dimensions \cite{BENBELGACEM1994235,MR1445738} and more recently \cite{MR2523203}. Finally we want to mention the works of \cite{MR2439500, MR2551195,MR2904580} where the polynomial extensions also provide a commuting diagram property for the corresponding spaces. 
\newline \newline For the proof of theorem \ref{ext:h2ext} we proceed in several steps. We start with an extension operator for the tangential values in section~\ref{ext::sectiontangext} to provide the properties \eqref{ext:h2ext:propA} and \eqref{ext:h2ext:propB}. After that we show in section~\ref{ext::sectionnormalgext} how to proceed with the normal component. For this we assume that the polynomial on the edge has a zero of order two in the vertices to show an estimate in the $H^{1/2}_{00}$-norm. To lift the normal trace also for arbitrary polynomials we show in section~\ref{ext::sectiondivide} that the error, hence the part of the polynomial which does not satisfy the assumptions needed for the extension before, is bounded with a proper order of $k$ to show \eqref{ext:h2ext:propB}. Finally we combine all estimates in section~\ref{ext::proofmaintheorem} to prove theorem \ref{ext:h2ext}. 
\begin{remark}
In the following proofs we only present the major techniques to show the $H^2$-continuity for the most difficult terms due to the similar structure of the rest. For a more detailed analysis we refer to \cite{lederer:2016}, chapter 6.
\end{remark}
\subsection{Tangential extension} \label{ext::sectiontangext}
\begin{theorem} \label{ext:tangh2ext}
  For every $k$ there exists an operator $\exttang : \bpolt^k(\Tref) \rightarrow \pol^{k+1}(\Tref)$ such that for $\vu  \in \bpolt^k(\Tref)$ it holds 
  \begin{align}
    &\nabla {\exttang(\vu)} \cdot \tau = \vu \cdot \tau \quad \textrm{on } \partial \Tref  \label{ext:h2exttang:propA} \\
    & \| \exttang(\vu) \|_{2,\Tref} \cle \| \vu \|_{1,\Tref} \label{ext:h2exttang:propC}
  \end{align}
\end{theorem}
\begin{proof}
  For the proof we proceed in several steps. First we construct an extension from the lower edge $E_1$ onto the triangle $\Tref$ without any restrictions on the values on the two other edges $E_2$ and $E_3$. After that we construct two more extensions with proper values on the other edges and combine them afterwards to define $\exttang$.
\begin{center}
\it Step 1
\end{center}
For $u_\tau(x) := \vu(x,0) \cdot \tau$, where $\tau := (1,0)^t$ is the tangential vector on the lower edge $E_1$, we define 
\begin{align} \label{ext::tangextdefstepone}
\psi(x) := \int_0^x u_\tau(s) \intd s \qquad \textrm{and} \qquad \exttang_1(\vu)(x,y) := \int_0^1 \psi(x + sy) \intd s.
\end{align}
Note that the derivations read as 
\begin{align*}
\exttang_{1,x}(\vu) = \int_0^1 \vut(x+sy) \intd s \quad \textrm{and} \quad  \exttang_{1,y}(\vu) = \int_0^1 \vut(x+sy)s \intd s,
\end{align*}
thus
\begin{align*}
  \nabla \exttang_1(\vu) \cdot \tau = \int_0^1 \vut(x) \intd s = \vu \cdot \tau \quad \textrm{on } E_1.
\end{align*}
Next we show the $H^2$ continuity, so $\| \exttang_1(\vu) \|_{0,\Tref}^2 + \| \nabla \exttang_1(\vu) \|_{1,\Tref}^2 \cle \| \vu \|_{1,\Tref}^2$. To derive the estimate of the $L^2$ norm we define for a fixed $y \in [0,1]$ the line $l_y := \{(x,y): x \in [0,1-y] \}$ and use the Cauchy Schwarz inequality to get
\begin{align*}
  \| \exttang_1(\vu) \|_{0,l_y}^2 &= \int_0^{1-y}\left(  \int_0^1 \psi(x + sy) \intd s\right)^2\intd x \cle \int_0^{1-y} \int_0^1 \psi(x + sy)^2 \intd s \intd x.                                 
\end{align*}
With the substitution $t = x+sy$ and Fubini's theorem this leads to
\begin{align*}
  \| \exttang_1(\vu) \|_{0,l_y}^2 & \cle \frac{1}{y} \int_0^{1-y} \int_x^{x+y} \psi(t)^2 \intd t \intd x = \frac{1}{y}\iint\limits_{\substack{0 \le x \le 1-y \\ x \le t \le x+y }} \psi(t)^2 \intd(x,t) \cle \frac{1}{y}\iint\limits_{\substack{0 \le t \le 1 \\ t-y \le x \le t }} \psi(t)^2 \intd(x,t) \\
  & = \frac{1}{y} \int_0^{1-y} \psi(t)^2 \underbrace{\int_{t-y}^t \intd x}_{=y} \intd t \cle \| \psi \|_{0,E_1}^2 \cle \| \vut \|_{0,E_1}^2 \cle \| \vut \|_{1/2,E_1}^2,
\end{align*}
thus
\begin{align*}
\| \exttang_1(\vu) \|_{0,\Tref}^2 = \int_0^1   \| \exttang_1(\vu) \|_{0,l_y}^2 \intd y \cle \| \vut \|_{1/2,E_1}^2  \cle  \| \vu \|_{1,\Tref}^2.
\end{align*}
The estimate of the first order derivations is similar, so we get $ \| \nabla \exttang_{1}(\vu) \|_{0,\Tref}^2 \cle \| \vu \|_{1,\Tref}^2$. For the rest, thus the estimate of the second order derivations, we use the real method of interpolation of spaces, see \cite{berghint} and Peetre's K-functional technique, see \cite{MR0178381}. It is well known that we have an equivalent norm on the space $H^{1/2}(E_1)$ given by
\begin{align*} 
\| \vut \|_{1/2,E_1}^2 = \int_0^\infty y^{-2} |K(y,\vut)|^2 \intd y \quad \textrm{with} \quad K(y,\vut) := \inf\limits_{\substack{\vu_0, \vu_1 \\\vut = (\vu_0 + \vu_1) \cdot \tau } } \sqrt{\|u_0 \cdot \tau\|_{0,E_1}^2 + y^2 \|u_1 \cdot \tau\|_{1,E_1}^2}.
\end{align*}
To use this method we first calculate the second order derivation with respect to $x$
\begin{align*}
\exttang_{1,xx}(\vu) = \int_0^1 \vut'(x+sy) \intd s. 
\end{align*}
The idea now is to find two different estimates of the $L^2$-norm of $\exttang_{1,xx}(\vu)$ and combine them using the definition of the $K$-functional. First we observe that again with similar techniques as in the beginning we bound the norm on $l_y$ by
\begin{align} \label{ext::exttangxxltwo}
\| \exttang_{1,xx}(\vu) \|_{0,l_y}^2 \cle  \| \vut' \|_{0,E_1}^2.
\end{align}
Next we use a different representation of $\exttang_{1,xx}(\vu)$ by using the identity $\vut'(x+sy) = \frac{1}{y} \frac{d}{ds} \vut(x+sy)$ and integration by parts
\begin{align*}
\exttang_{1,xx}(\vu)  = \int_0^1 \vut'(x+sy) \intd s = \underbrace{\frac{1}{y} \int_0^1 \vut(x+sy) \intd s}_{=:A_1} + \underbrace{\frac{1}{y} ( \vut(x+y) - \vut(x))}_{=:B_1}.
\end{align*}
For $A_1$ we proceed as before to get $\| A_1 \|_{0,l_y}^2 \cle \frac{1}{y^2} \| \vut \|^2_{0,E_1}$, and for $B_1$ we get with $t= x+y$
\begin{align*}
  \| B_1 \|_{0,l_y}^2 &= \frac{1}{y^2} \int_0^{1-y} ( \vut(x+y) - \vut(x))^2 \intd x \cle \frac{1}{y^2} \int_0^{1-y} \vut(x+y)^2 + \vut(x)^2 \intd x \\
                    &\cle \frac{1}{y^2} \int_y^{1} \vut(t)^2 \intd t  + \frac{1}{y^2}\int_0^{1-y}  \vut(x)^2 \intd x \cle \frac{1}{y^2} \| \vut \|^2_{0,E_1}.
\end{align*}
Together with \eqref{ext::exttangxxltwo} we have
\begin{align*}
  \| \exttang_{1,xx}(\vu)  \|^2_{0,l_y} &\le \inf\limits_{\substack{\vu_0, \vu_1 \\\vut = (\vu_0 + \vu_1) \cdot \tau } } \| \exttang_{1,xx}(\vu_0)  \|^2_{0,l_y} + \| \exttang_{1,xx}(\vu_1)  \|^2_{0,l_y} \\
  &\cle \inf\limits_{\substack{\vu_0, \vu_1 \\\vut = (\vu_0 + \vu_1) \cdot \tau } } \frac{1}{y^2} \| u_0 \cdot \tau  \|^2_{0,E_1} + \| (u_1 \cdot \tau)' \|^2_{0,E_1}   \cle \frac{1}{y^2} K(y,\vut)^2, 
\end{align*}
and thus
\begin{align*}
  \| \exttang_{1,xx}(\vu)  \|_{0,\Tref}^2  &= \int_0^1 \| \exttang_{1,xx}(\vu)  \|_{0,l_y}^2 \intd y \cle \int_0^1 \frac{1}{y^2} K(y,\vut)^2 \intd y \\
  & \cle \int_0^\infty \frac{1}{y^2} K(y,\vut)^2 \intd y = \| \vut \|_{1/2,E_1}^2.
\end{align*}
For the other second order derivations we proceed similarly. All together we have
\begin{align} \label{ext::tangextsteponeest}
  \| \exttang_{1}(\vu) \|_{2,\Tref}^2 \cle \| \vut \|_{1/2,E_1}^2 \cle \| \vu \|_{1,\Tref}^2.
\end{align}
\begin{center}
\it Step 2
\end{center}
In the next step we want to construct an extension from the lower edge $E_1$ with the restriction that the gradient of the extension has zero tangential values on the edge $E_2$. Similar to before we define for $u_\tau(x) := \vu(x,0) \cdot \tau$ on $E_1$ 
\begin{align*}
\psi(x) := \int_0^x u_\tau(s) \intd s - \overline{\psi(x)} \qquad \textrm{with} \qquad \overline\psi(x) := \int_0^1 u_\tau(s) \intd s,
\end{align*}
and
\begin{align} \label{ext::tangextdefsteptwo}
  \exttang_2(\vu)(x,y) := \int_0^1 \psi(x + sy) \intd s - \frac{y}{1-x} \int_0^1 \psi(x+s(1-x)) \intd s.
\end{align}
Using integration by parts for the second integral we furthermore use the representation
\begin{align} \label{ext::tangextdefsteptwoalt}
  \exttang_2(\vu)(x,y) &= \int_0^1 \psi(x + sy) \intd s + y \int_0^1 \psi'(x+s(1-x)) s \intd s - \overbrace{\psi(1)}^{=0} \nonumber\\
                       &= \int_0^1 \psi(x + sy) \intd s + y \int_0^1 u_\tau(x+s(1-x)) s \intd s.
\end{align}
On the edges $E_1, E_2$ we have
\begin{align*}
\exttang_2(\vu)|_{E_2} = 0 \Rightarrow \nabla \exttang_2(\vu) \cdot \tau|_{E_2} = 0 \quad \textrm{and} \quad \exttang_2(\vu) \cdot \tau|_{E_1} = \frac{\partial \psi}{\partial x} = u_\tau = \vu \cdot \tau|_{E_1}.
\end{align*}
For the $H^2$ continuity we proceed as before. Note that the first part of $\exttang_2(\vu)$ is the same as in step 1, so we only consider the second integral from \eqref{ext::tangextdefsteptwoalt}, thus the correction term
\begin{align*}
 \exttangc_2(\vu) := y \int_0^1 u_\tau(x+s(1-x)) s \intd s.
\end{align*}
We present only the estimate for the second order derivative with respect to $x$ as some new techniques have to be used there. The rest follows with similar estimates. Using integration by parts for $\vut'$ we observe
\begin{align*}
   \exttangc_{2,xx}(\vu) &= \frac{y}{(1-x)^2} \int_0^1 \vut(x+s(1-x))(2s-1) \intd s \\
                                                                                        & +  \frac{y}{1-x} \int_0^1 \vut'(x+s(1-x))(2s-1) \intd s \\
                                                                                        &= \frac{y}{(1-x)^2} \int_0^1 \vut(x+s(1-x))(2s-1) \intd s \\
                                                                                        & +  \frac{y}{(1-x)^2} \int_0^1 \vut(x+s(1-x))(4s-3) \intd s  + \frac{y}{(1-x)^2} \vut(x),
\end{align*}
and by splitting the second integral into two terms finally
\begin{align*}
  \exttangc_{2,xx}(\vu) = \underbrace{\frac{y}{(1-x)^2} \int_0^1 \vut(x+s(1-x))(6s-3) \intd s}_{=:A_2} + \underbrace{ \frac{y}{(1-x)^2} \int_0^1 \vut(x) - \vut(x+s(1-x)) \intd s}_{=: B_2}.
\end{align*}
We start with the estimate of $A_2$. For this note that the polynomial $6s-3$ has a zero integral value on $E_1$, so we subtract the mean value $\overline{\vut(x)} := \frac{1}{1-x} \int_x^1 \vut(s) \intd s$, and get
\begin{align*}
  A_2 &= \frac{y}{(1-x)^2} \int_0^1 \vut(x+s(1-x))(6s-3) \intd s = \frac{y}{(1-x)^2} \int_0^1 (\vut(x+s(1-x)) - \overline{\vut(x)}) (6s-3) \intd s.
\end{align*}
Using the Cauchy Schwarz inequality and $\frac{y}{1-x} \le 1$ on $\Tref$ and the substitution $t = x + s(1-x)$ leads to 
\begin{align*}
  \| A_2 \|^2_{0,\Tref} &= \int_0^1 \int_0^{1-x} A_2^2 \intd y \intd x \\
                      & \cle  \int_0^1 \frac{1}{(1-x)^2} \left(  \int_0^1 (\vut(x+s(1-x)) - \overline{\vut(x)}) (6s-3) \intd s \right)^2 \int_0^{1-x}  \intd y \intd x  \\
                      & \cle \int_0^1  \frac{1}{1-x} \int_0^1 (\vut(x+s(1-x)) - \overline{\vut(x)})^2 \intd s \intd x \\
  & =  \int_0^1  \frac{1}{(1-x)^2} \int_x^1 (\vut(t) - \overline{\vut(x)})^2 \intd t \intd x. 
\end{align*}
For the next step we use the following identity for the inner integral
\begin{align*}
\int_x^1 (\vut(t) - \overline{\vut(x)})^2 \intd t  = \frac{1}{2(1-x)} \int_x^1 \int_x^1 (\vut(t) - \vut(s))^2 \intd t \intd s.
\end{align*}
Similar to step 1 we use now Fubini's theorem to handle the $(1-x)^3$ in the denominator, so
\begin{align*}
  \| A_2 \|^2_{0,\Tref} &\cle \int_0^1 \frac{1}{(1-x)^3}  \int_x^1 \int_x^1 (\vut(t) - \vut(s))^2 \intd t \intd s \intd x \\
                      & \cle \iiint\limits_{\substack{x \le s\\ x \le t \\ 0 \le s,t \le 1 }} \frac{(\vut(t) - \vut(s))^2 }{(1-x)^3} \intd (s,t,x) \cle \int_0^1 \int_0^1 \int_0^{\min(s,t)} \frac{1}{(1-x)^3} \intd x (\vut(t) - \vut(s))^2 \intd t \intd s.
\end{align*}
W.l.o.g assuming $s < t$ and using $1-s \ge t-s$ for $s < t < 1$, we see for the inner integral that 
\begin{align*}
 \int_0^{\min(s,t)} \frac{1}{(1-x)^3}\intd x = \left( \frac{1}{1-\min(s,t) } \right) ^2 - 1 \le \left( \frac{1}{1-s} \right)^2 \le \left ( \frac{1}{t-s} \right)^2,
\end{align*}
so together with the definition of the $H^{1/2}$ seminorm (see Sobolev Slobodeckij norm for example in \cite{opac-b1127030}, theorem A.7) we get
\begin{align*}
\| A_2 \|^2_{0,\Tref} & \cle \int_0^1 \int_0^1 \frac{(\vut(t) - \vut(s))^2}{(t-s)^2} \intd t \intd s = | \vut|^2_{1/2,E_1} \cle \| \vu \|^2_{1,\Tref}.
\end{align*}
For $B_2$ we proceed similarly using the Cauchy Schwarz inequality, Fubini's theorem, the substitution $t = x + s(1-x)$, and $1-x \ge t-x$ for $x < t < 1$, thus
\begin{align*}
  \| B_2 \|^2_{0,\Tref} &= \int_0^1 \int_0^{1-x} B_2^2 \intd y \intd x \cle \int_0^1 \frac{1}{1-x} \int_0^1  \left( \vut(x) - \vut(x+s(1-x)) \right )^2 \intd s \intd x \\
                       & =  \int_0^1 \frac{1}{(1-x)^2} \int_x^1  \left( \vut(x) - \vut(t) \right)^2 \intd t \intd x  = \iint\limits_{\substack{x \le t \le 1\\0\le x \le 1 }} \frac{\left( \vut(x) - \vut(t) \right)^2}{(1-x)^2} \intd(x,t)\\
                       &\cle \iint\limits_{\substack{0 \le x \le t \\0\le t \le 1 }} \frac{\left( \vut(x) - \vut(t) \right)^2}{(x-t)^2} \intd(x,t) \cle \int_0^1 \int_0^1 \frac{ \left( \vut(x) - \vut(t) \right)^2}{(x-t)^2} \intd x \intd t \cle  | \vut|^2_{1/2,E_1} \cle \| \vu \|^2_{1,\Tref},
\end{align*}
so we have $ \| \exttangc_{2,xx}(\vu)\|_{0,\Tref} \cle \| \vu \|_{1,\Tref}$ and assuming similar estimates for the other derivatives all together the $H^2$-continuity
\begin{align} \label{ext::exttangtwoest}
 \| \exttang_{2}(\vu)\|_{2,\Tref} \cle \| \vu \|_{1,\Tref}.
\end{align}
\begin{center}
\it Step 3
\end{center}
For this step we assume that the tangential values of the input function $\vu$ is zero on the edges $E_2$ and $E_3$ and that it has a zero tangential integral value, thus
\begin{align} \label{ext::assumptionsstepthree}
\int_{\partial \Tref} \vu \cdot \tau = 0 \quad \textrm{and} \quad \vu \cdot \tau|_{E_2} = \vu \cdot \tau|_{E_3} = 0.
\end{align}
We set  $u_\tau(x) := \vu(x,0) \cdot \tau$ on $E_1$ and $ \psi(x) := \int_0^x u_\tau(s) \intd s$, to define
\begin{align} \label{ext::tangextdefstepthree}
  \exttang_{3}(\vu) &:= \int_0^1 \psi(x+sy) \intd s - \frac{y}{1-x} \int_0^1 \psi(x+s(1-x)) \intd s \nonumber\\
  & \underbrace{- \frac{y}{x+y} \int_0^1 \psi(s(x+y)) \intd s}_{=: A_3} + \underbrace{ y \int_0^1 \psi(s) \intd s}_{=: B_3}.
\end{align}
As in step 2 we observe $\nabla \exttang_{3}(\vu) \cdot \tau |_{E_2} = \nabla \exttang_{3}(\vu) \cdot \tau |_{E_3} = 0$ and $\nabla \exttang_{3}(\vu) \cdot \tau |_{E_1} = \vut$. For the $H^2$-continuity we only have to estimate the terms $A_3$ and $B_3$ as the other terms are the same as in step~1 and step~2. For this note that due to the assumptions on $\vu$, we have $\psi(1) = \psi(0) = 0$, so using the identity $\frac{\intd}{\intd s} \psi(s(x+y)) \frac{1}{x+y} = \psi'(s(x+y))$ and integration by parts we write
\begin{align} \label{ext::tangextdefstepthreealtone}
  A_3 &= - \frac{y}{x+y} \int_0^1 \psi(s(x+y)) \intd s  = y \int_0^1 \psi'(s(x+y))(s-1) \intd s = y \int_0^1 \vut(s(x+y))(s-1) \intd s,
\end{align}
and
\begin{align} \label{ext::tangextdefstepthreealttwo}
B_3 =  y \int_0^1 \psi(s) \intd s = -y  \int_0^1 \psi'(s)s \intd s =  -y \int_0^1 \vut(s)s \intd s.
\end{align}
Using these representations and the same techniques as in step 2 and step 3 we estimate the $H^2$-norm of $A_3$ and $B_3$ to show
\begin{align} \label{ext::exttangthreeest}
 \| \exttang_{3}(\vu) \|_{2,\Tref} \cle \| \vu \|_{1,\Tref}.
\end{align}
\begin{center}
\it Step 4
\end{center}
We finally combine the three extensions to show theorem \ref{ext:tangh2ext}. For that assume we have a given function $\vu \in \bpolt^k(\Tref)$ with $\int_{\partial \Tref} \vu \cdot \tau = 0$. We first introduce two mappings from the reference triangle $\Tref$ to itself by
\begin{align*}
F_2: (x,y) \mapsto (x,1-x-y) \quad \textrm{and} \quad F_3: (x,y) \mapsto (y,x),
\end{align*}
where $F_2$ maps the values from $E_2$ to $E_1$ and vice versa, and the mapping $F_3$ from $E_3$ to $E_1$ and vice versa. Furthermore we define for $F_2$ and $F_3$ the corresponding covariant mappings $\cov_2$ and $\cov_3$. Using those transformation we now introduce the extensions from step 2 and step 3 also from the other edges, thus we define
\begin{align*}
\tilde{\exttang_2}(\vu)(x,y) = \exttang_2(\cov_2\vu)(F_2(x,y)) \quad \textrm{and} \quad \tilde{\exttang_3}(\vu)(x,y) = \exttang_3(\cov_3\vu)(F_3(x,y)),
\end{align*}
with the properties
\begin{align}
  &(\nabla \tilde{\exttang_2}(\vu)\cdot \tau)|_{E_2} = (\vu \cdot \tau)|_{E_2}  \quad  (\nabla \tilde{\exttang_2}(\vu)\cdot \tau)|_{E_1} = 0  \quad  \| \tilde{\exttang_2}(\vu) \|_{2,\Tref} \cle  \| \vu \|_{1,\Tref}, \label{ext::exttangtwotildeest}\\
    &(\nabla \tilde{\exttang_3}(\vu)\cdot \tau)|_{E_3} = (\vu \cdot \tau)|_{E_3}  \quad  (\nabla \tilde{\exttang_3}(\vu)\cdot \tau)|_{E_1} = (\nabla \tilde{\exttang_3}(\vu)\cdot \tau)|_{E_2}= 0  \quad  \| \tilde{\exttang_3}(\vu) \|_{2,\Tref} \cle  \| \vu \|_{1,\Tref},\label{ext::exttangthreetildeest}
\end{align}
what follows from the proper transformation of tangential values due to the use of the covariant transformations $\cov_2$ and $\cov_3$ and estimates \eqref{ext::exttangtwoest} and \eqref{ext::exttangthreeest}. We define the final extension by setting $e_1 := \exttang_1(\vu)$, $e_2 := e_1 + \tilde{\exttang_2}(\vu- \nabla e_1)$ and $\exttang(\vu) := e_2 + \tilde{\exttang_3}(\vu - \nabla e_2)$. Note that due to Green's theorem the surface integral over the boundary of the reference element $\partial \Tref$ of $(\vu - \nabla e_2) \cdot \tau$ is equal to zero, and due to the properties of $\tilde{\exttang_2}$ and $\exttang_1$, also the tangential values on $E_1$ and $E_2$ vanish, thus assumptions \eqref{ext::assumptionsstepthree} are fulfilled. Using \eqref{ext::exttangtwotildeest} and \eqref{ext::exttangthreetildeest} we observe
\begin{align*}
  \nabla \exttang(\vu) \cdot \tau|_{E_1} &= \nabla e_2 \cdot \tau|_{E_1} + \underbrace{\nabla \tilde{\exttang}_3(\vu - \nabla e_2) \cdot \tau|_{E_1}}_{=0} = \nabla e_1 \cdot \tau|_{E_1} + \underbrace{ \nabla \tilde{\exttang}_2(\vu - \nabla e_1) \cdot \tau|_{E_1}}_{=0} =  (\vu \cdot \tau)|_{E_1},
\end{align*}
and similarly $\nabla \exttang(\vu) \cdot \tau|_{E_2} =  \vu \cdot \tau|_{E_2}$ and $\nabla \exttang(\vu) \cdot \tau|_{E_3} = \vu \cdot \tau|_{E_3}$, 
thus property \eqref{ext:h2exttang:propA} is fulfilled. With \eqref{ext::exttangtwotildeest}, \eqref{ext::exttangthreetildeest} and \eqref{ext::tangextsteponeest} and the linearity of the operators we furthermore have $\| \exttang(\vu) \|_{2,\Tref} \cle  \| \vu \|_{1,\Tref}$ thus also the $H^2$-continuity \eqref{ext:h2exttang:propC} holds true. It remains to show that $\exttang(\vu) \in \pol^{k+1}(\Tref)$. First note that from $\vu \in [\pol^k(\Tref)]^2$ it follows that $\vu \cdot \tau \in \pol^k(\partial \Tref)$. Looking at the definition of the first extension \eqref{ext::tangextdefstepone} we see that we integrate $u$ from $0$ to $x$ to define $\psi$, thus here we increase the order by one resulting in $\exttang_1(\vu) \in \pol^{k+1}(\Tref)$. For the other two  extensions \eqref{ext::tangextdefsteptwo} and \eqref{ext::tangextdefstepthree} this may not hold due to the fractional factors, but using the alternative representations \eqref{ext::tangextdefsteptwoalt}, \eqref{ext::tangextdefstepthreealtone}, \eqref{ext::tangextdefstepthreealttwo} we see that also the corrections are polynomial liftings thus we have $\exttang(\vu) \in \pol^{k+1}(\Tref)$.
\end{proof}
\subsection{Normal extension} \label{ext::sectionnormalgext}
\begin{theorem} \label{ext:normh2ext}
   For every $k$ there exists an operator $\extnorm : \pol_{00}^k(E_1) \rightarrow \pol^{k+1}(\Tref)$ such that for $u \in \pol_{00}^k(E_1)$ it holds 
    \begin{align}
      &\extnorm(u) = 0  \quad \textrm{on } \partial \Tref  \label{ext:h2extnorm:propA} \\
      &\nabla \extnorm(u) \cdot \vn = u  \quad \textrm{on } E_1  \label{ext:h2extnorm:propB} \\
      &\nabla \extnorm(u) \cdot \vn = 0  \quad \textrm{on } \partial \Tref \setminus  E_1  \label{ext:h2extnorm:propD} \\
      & \| \extnorm(u) \|_{2,\Tref} \cle \| u \|_{1/2^*,E _1} \label{ext:h2extnorm:propC}
    \end{align}          
\end{theorem}
\begin{proof}
  Similar to the proof of theorem \ref{ext:tangh2ext} we proceed in several steps. We first construct an extension from the lower edge and correct the value and the normal derivative on the other two edges afterwards.
\begin{center}
\it Step 1
\end{center}
We start with the first extension, so we define
\begin{align} \label{ext::extnormstepone}
 \extnorm_1(u) := -y \int_0^1a(s)u(x+sy) \intd s \quad \textrm{with} \quad a(s):=6s(1-s).
\end{align}
It immediately follows $\extnorm_1(u)|_{E_1} = 0$. For the derivations we observe due to $a(0) = a(1) = 0 $, $\int_0^1a(s) \intd s = 1$, $\int_0^1a'(s)s \intd s = -1$ and using integration by parts with $u'(x+sy) = \frac{1}{y}\frac{d}{ds} u(x+sy)$, that
\begin{align*}
\extnorm_{1,x}(u) &=  -y \int_0^1 a(s) u'(x+sy) \intd s =  \int_0^1 a'(s) u(x+sy) \intd s \\
   \extnorm_{1,y}(u) &= -\int_0^1 a(s) u(x+sy) \intd s  - y\int_0^1 a(s) u'(x+sy)s \intd s = \int_0^1 a'(s) u(x+sy)s \intd s.                                              
\end{align*}
thus
\begin{align*}
  \nabla \extnorm_1(u) \cdot n|_{E_1} =- \extnorm_{1,y}(u)|_{E_1} =  u.
\end{align*}
The $H^2$-continuity estimate follows with the K-functional technique presented in step 1 in the proof of theorem \ref{ext:tangh2ext} for the derivations and for the rest by the Cauchy Schwarz inequality, thus we have
\begin{align*}
\| \extnorm_1(u) \|_{2,\Tref} \cle \| u \|_{1/2,E_1} \cle \| u \|_{1/2^*,E_1} 
\end{align*}
\begin{center}
\it Step 2
\end{center}
In this step we want to correct the values and the normal derivative on the second edge $E_2$ without changing the values and the normal derivative on the bottom edge $E_1$. For this we introduce the polynomials $b(s) = 3s^2-2s^3$ and $c(s)=s^3-s^2$, with the properties
\begin{align*}
  b(0) = b'(0) = b'(1) = 0, b(1) = 1\quad \textrm{and} \quad   c(0) = c'(0) = c(1) = 0, c'(1) = 1,
\end{align*}
and use them as blending coefficients to define
\begin{align*}
 \extnorm_2(u)(x,y):=  \extnorm_1(u)(x,y) - b(\frac{y}{1-x})\extnorm_1(u) (x,1-x) - c(\frac{y}{1-x})(1-x) \extnorm_{1,y}(u)(x,1-x).
\end{align*}
The idea is that the second term corrects the values and the last term corrects the normal derivation on the edge $E_2$. Indeed we observe on $E_1$ as $y=0$ and on $E_2$ as $y = 1-x$ that
\begin{align}
  \extnorm_2(u)|_{E_1}&= \extnorm_1(u)|_{E_1}- b(0) \extnorm_1(u)(x,1-x) - c(0) \extnorm_{1,y}(u)(x,1-x) = 0,\nonumber \\ 
  \extnorm_2(u)|_{E_2}&= \extnorm_1(u)|_{E_1}- b(1) \extnorm_1(u)(x,1-x) - c(1) \extnorm_{1,y}(u)(x,1-x) = 0.\label{ext::extnormstep2onedgetwoB}
\end{align}
Next we look at the derivation with respect to $y$ using the chain and product rule
\begin{align*}
\extnorm_{2,y}(u) &=  \extnorm_{1,y}(u) - \frac{1}{1-x}b'(\frac{y}{1-x}) \extnorm_1(u) (x,1-x)  - \frac{1}{1-x}c'(\frac{y}{1-x})(1-x) \extnorm_{1,y}(u) (x,1-x).
\end{align*}
We see that
\begin{align*}
 \extnorm_{2,y}(u)|_{E_1}= \extnorm_{1,y}(u)|_{E_1} - \frac{b'(0)}{1-x} \extnorm_{1}(u)(x,1-x) - c'(0) \extnorm_{1,y}(u)(x,1-x) = \extnorm_{1,y}(u)|_{E_1} = u
\end{align*}
thus the normal derivative $\nabla  \extnorm_{2}(u) \cdot n|_{E_1}= \nabla  \extnorm_{1}(u) \cdot n|_{E_1} = u$ has not changed in the second step. In a similar way we also observe that $\extnorm_{2,y}(u)|_{E_2}=0$. Now note that due to the constant zero value on the edge, see equation \eqref{ext::extnormstep2onedgetwoB}, we derive that the tangential derivation on the edge $\nabla \extnorm_2(u) \cdot \tau$ on $E_2$ has to be zero. As $\nabla \extnorm_2(u) \cdot \tau|_{E_2} = 0 \Leftrightarrow  - \extnorm_{2,x}(u)|_{E_2} = \extnorm_{2,y}(u)|_{E_2}$ and $\extnorm_{2,y}(u)|_{E_2} = 0$, it follows $\nabla  \extnorm_{2}(u) \cdot n|_{E_2} = 0$, so the correction term induced a zero normal derivative on the second edge $E_2$. It remains the $H^2$-estimate. As the first term of $ \extnorm_2(u)$ was already analysed in the first step, we just focus on the correction terms. We start with first term $A_4:=  b(\frac{y}{1-x})\extnorm_1(u) (x,1-x)$ and the estimate for the $y$ derivative
\begin{align*}
A_{4,y} = 6\frac{y}{1-x}\left(1-\frac{y}{1-x}\right) \int_0^1 u(x+s(1-x))a(s) \intd s.
\end{align*}
Using the Cauchy Schwarz inequality, $\frac{y}{1-x} \le 1$ on $\Tref$ and the substitution $t := x+s(1-x)$ we get
\begin{align*}
  \| A_{4,y} \|_{0,\Tref}^2 &\cle \int_0^1\int_0^{1-x}\int_0^1 u(x+s(1-x))^2  \intd s \intd y \intd x = \int_0^1\int_0^{1-x} \frac{1}{1-x}\int_x^1 u(t)^2  \intd t \intd y \intd x \\
  &= \int_0^1 \int_x^1 u(t)^2 \intd t \intd x \cle \| u \|_{0,E_1}^2 \cle \| u \|_{1/2^*,E_1}^2,
\end{align*}
and in a similar way we also bound $\| A_{4,x} \|_{0,\Tref}$ and $\| A_4 \|_{0,\Tref}$. The crucial point in this estimate was, that we were able to use property $\frac{y}{1-x} \le 1$ twice, thus there were no {\it bad} coefficients anymore. The estimates of the second order derivatives are a little bit more tricky as there remain some fractions with singularities. We start with the second order derivation with respect to $y$ given by
\begin{align*}
 A_{4,yy} = 6\frac{1}{1-x}\left(1-\frac{2y}{1-x}\right) \int_0^1 u(x+s(1-x))a(s) \intd s.
\end{align*}
The idea is to use Fubini's theorem
\begin{align*}
  \| A_{4,yy} \|_{0,\Tref}^2 &\cle \int_0^1\int_0^{1-x} \frac{1}{(1-x)^2}\int_0^1 u(x+s(1-x))^2  \intd s \intd y \intd x \\
                           &= \int_0^1 \frac{1}{(1-x)^2}\int_x^1 u(x+s(1-x))^2  \intd s \intd y \intd x = \iint\limits_{\substack{0 \le x \le 1\\ x \le t}} \frac{u(t)^2}{(1-x)^2}\intd(x,t) \\
                           &= \iint\limits_{\substack{0 \le t \le 1\\ t \le x}} \frac{u(t)^2}{(1-x)^2}\intd(x,t) = \int_0^1\int_0^t \frac{1}{(1-x)^2} \intd x~ u(t)^2 \intd t \\
  &= \int_0^1 \frac{1}{1-t}u(t)^2 \intd t \cle  \| u \|_{0^*,E_1}^2 \cle \| u \|_{1/2^*,E_1}^2. 
\end{align*}
With the techniques just presented and the techniques from the proof in theorem \ref{ext:tangh2ext} all other derivations of the second correction are bounded and we get
\begin{align*}
\| \extnorm_2(u) \|_{2,\Tref}  \cle \| u \|_{1/2^*,E_1} 
\end{align*}
\begin{center}
\it Step 3
\end{center}
Similar to step 2 we correct now the values on the last edge $E_3$ with two more corrections using the same blending coefficients, thus we define
\begin{align*}
  \extnorm(u)(x,y):=&  \extnorm_2(u)(x,y) - b(\frac{y}{x+y})\extnorm_2(u) (0,x+y) \\
  &- c(\frac{y}{x+y})(x+y) \extnorm_{2,x}(u)(0,x+y) +c(\frac{y}{x+y})(x+y) \extnorm_{2,y}(u)(0,x+y).
\end{align*}
With the same arguments and estimates as in step 2 it follows \eqref{ext:h2extnorm:propA},\eqref{ext:h2extnorm:propB},\eqref{ext:h2extnorm:propD} and \eqref{ext:h2extnorm:propC}
\begin{gather*}
  \extnorm(u)|_{E_1} =   \extnorm(u)|_{E_2} =   \extnorm(u)|_{E_3} = 0 \\
  (\nabla \extnorm(u)\cdot n)|_{E_2} = (\nabla  \extnorm(u)\cdot n)|_{E_3}  = 0 \quad \textrm{and} \quad (\nabla \extnorm(u)\cdot n)|_{E_1} = u \\
\| \extnorm(u) \|_{2,\Tref}  \cle \| u \|_{1/2^*,E_1}.
\end{gather*}
It remains to show that $\extnorm(u)$ belongs to $\pol^{k+1}(\Tref)$. The idea is similar to the tangential extension. Looking at the definition of the first step \eqref{ext::extnormstepone} we increase the order by multiplying with $y$. The crucial parts are the correction terms as the blending polynomials $b(y/(1-x))$ and $c(y/(1-x))$ for the second step, and $b(y/(x+y))$ and $c(y/(x+y))$ for the third step produce singularities of order three in the points $(0,0)$ and $(1,0)$. To overcome this problem note that the given polynomial has a zero of order two in the vertices, thus there exists a polynomial $v \in \pol^{k-2}(E_1)$ such that $u(x) = (1-x)^2v(x)$. Using the definitions of the polynomials $b$ and $c$ the extension of the second step  $\extnorm_2(u)$ reads as
\begin{align*}
  \extnorm_2(u)(x,y) = y \int_0^1 a(s)u(x+sy) \intd s &- \frac{3 y^2(1-x)-2y^3}{(1-x)^2}\int_0^1a(s) u(x+s(1-x)) \intd s \\
  & - \frac{y^3-y^2(1-x)}{(1-x)^2}\int_0^1a'(s) u(x+s(1-x))s \intd s.
\end{align*}
As $u(x+s(1-x)) = (1-x)^2(1-s)^2v(x+s(1-x))$ this leads to
\begin{align*}
  \extnorm_2(u)(x,y) = y \int_0^1 a(s)u(x+sy) \intd s &- (3 y^2(1-x)-2y^3)\int_0^1a(s) (1-s)^2v(x+s(1-x)) \intd s \\
  & - (y^3-y^2(1-x))\int_0^1a'(s) (1-s)^2v(x+s(1-x))s \intd s,
\end{align*}
thus $\extnorm_2(u) \in \pol^{k+1}(\Tref)$. For the third step we do the same by writing $u(x) = x^2w(x)$ with $w \in \pol^{k-2}(E_1)$ finally leading to $\extnorm(u) \in \pol^{k+1}(\Tref)$.
\end{proof}
\begin{remark} \label{remark::normalext}
In a similar way as in the last step of the proof of theorem \ref{ext:tangh2ext} it is possible to define the normal extension $\extnorm$ for the other two edges $E_2$ and $E_3$ by using proper transformations. We then use the subscript $\extnorm_{E_i}(\cdot)$ with $i \in \{ 1,2,3 \}$ to symbolize which extension is used.
\end{remark}
\subsection{Splitting into compatible and incompatible polynomials} \label{ext::sectiondivide}
Using theorem \ref{ext:normh2ext} it would now be possible to correct the normal derivative after a first extension using theorem \ref{ext:tangh2ext}. The crucial point is that the polynomial would need a zero of order two in the vertices. The following theorem helps us later to provide a stable splitting of the correction into two parts.
\begin{theorem} \label{ext:dividetheorem}
  Assume a given function $u \in \pol^k(E_1)$, with $u =0 $ on $\partial E_1$. Then it holds
  \begin{align} \label{ext::divtheorem:A}
    |u'(1)| \cle k^2 \| u \|_{1/2^*,E_1}.
  \end{align}  
  Furthermore there exists a function $e \in \pol^k(E_1)$ with $e'(1) = 1$ and $e'(0) = e(0) = e(1) = 0$ such that
    \begin{align} \label{ext::divtheorem:B}
      \| e\|_{1/2^*,E_1} \cle \frac{1}{k^2} \quad \textrm{and} \quad \| e\|_{0,E_1} \cle \frac{1}{k^3}.
    \end{align}
\end{theorem}
\begin{proof}
  We start with the second statement. For this we present the proof on the interval $E$ from which the original statement follows with a linear transformation. We want to remind the reader of  the definition of Jacobi polynomials with respect to the weight function $(1-x)^\alpha$, see for example in \cite{Abramowitz} or \cite{andrews}, 
  \begin{align*}
p^\alpha_n(x) := \frac{1}{2^nn!(1-x)^\alpha} \frac{d}{d x^n}\left( (1-x)^\alpha(x^2-1)^n \right) \quad n \in \N_0, \alpha > -1,
  \end{align*}
  where in the special case of $\alpha = 0$ the polynomials are called Legendre polynomials. For our proof we use integrated Jacobi polynomials with $\alpha=1$
  \begin{align*}
    \intjacobi_n(x) &:= -\int_x^1 p_{n-1}^1(s) \intd s \quad  n\ge 1 \quad \textrm{and} \quad \intjacobi_0 (x) := 1,
  \end{align*}
  and integrated Legendre polynomials
  \begin{align*}
    l_{n+1}(x) := -\int_x^1 p_n^0(s) \intd s \quad  n \ge 0 \quad \textrm{and} \quad l_0 := -x + 1.
  \end{align*}
 It holds the following properties, see \cite{andrews} and \cite{MR2221053}, 
 \begin{gather}
    \intjacobi_n (1) = 0 \quad 1 \le n \le k \quad \textrm{and} \quad  \intjacobi'_n (1) = n \quad 0 \le n \le k \label{ext::polpropC}\\
    (2n+1) p_n^0 = (n+1) p^1_n - n p^1_{n-1} \quad n\ge 0 \quad \textrm{and} \quad  p_m^1 = \frac{1}{m+1} \sum\limits_{n=0}^m (2n+1) p_n^0 \quad m\ge 0,\label{ext::polpropE} \\
           (2n+1) l_{n+1} = p^0_{n+1} - p^0_{n-1} \quad n>0,\label{ext::polpropF}
 \end{gather}
 where we used $p^0_{-1} := -1$. Furthermore we have a weighted $L^2$ orthogonality for the Jacoby polynomials, and due to the definition a weighted orthogonality in the $H^1$ seminorm for the integrated Jacoby polynomials
 \begin{align} \label{ext::dividingtheorem::orthogonality}
\int_{-1}^1 (1-x) \intjacobi'_n(x) \intjacobi'_m(x) \intd x = \int_{-1}^1 (1-x) p^1_{n-1}(x) p^1_{m-1}(x) \intd x = \delta_{n,m} \frac{2}{n+1},
 \end{align}
 where $\delta$ is the Kronecker delta.
 To find a proper candidate which fulfills the bounds \eqref{ext::divtheorem:B} we first seek for the minimum of the weighted $H^1$ seminorm with proper restrictions, thus
\begin{align*}
\tilde{e} := \arg \min\limits_{\substack{v \in \pol^k \\ v(1) = 0 \\v'(1) = 1}}  \int_{-1}^1(1-x)v'(x)^2 \intd s = \arg \min\limits_{\substack{v \in \pol^k \\ v(1) = 0 \\v'(1) = 1}} | v|^2_{1^*,E}.
\end{align*}
Using integrated Jacobi polynomials as basis for $\pol^k(E)$ we use the representation of $\tilde{e}$ with coefficients $c_j$ as $\tilde{e}(x) = \sum\limits_{j=0}^k c_j \intjacobi_j(x)$. To determine the coefficients, so to explicitly solve the minimization problem, we first note that due to the boundary restrictions $\tilde{e}(1) = 0$ it is clear that $c_0 =0$ and with \eqref{ext::polpropC} we get $\tilde{e}'(1) = \sum\limits_{j=1}^k c_j j = 1$. Using \eqref{ext::dividingtheorem::orthogonality} we furthermore have $| \tilde{e}|^2_{1^*,E} = \sum\limits_{j=1}^k c_j^2\frac{2}{j+1}$.  Now we use the technique of Lagrangian multipliers, thus we define the function 
\begin{align*}
  L(c_1,\dots,c_k,\lambda) = \sum\limits_{j=1}^k c_j^2\frac{2}{j+1} +  \lambda (\sum\limits_{j=1}^k c_j j -1) \quad \textrm{with} \quad \frac{\partial L}{ \partial c_j} \stackrel{!}{=} 0 \quad  \forall j \in \{1,\dots,k \} \quad \textrm{and} \quad \frac{\partial L}{\partial  \lambda} \stackrel{!}{=} 0.
\end{align*}
Solving this leads to
\begin{align} 
\lambda = \frac{-48}{k(k+1)(k+2)(3k+1)}  \quad \textrm{and} \quad c_j= \frac{-12 j(1+j)}{k(k+1)(k+2)(3k+1)},\label{ext::divi::coeffs}
\end{align}
and
\begin{align} \label{ext:divtheorem:approx}
| \tilde{e} |^2_{1*,E} = \sum\limits_{j=1}^k c_j^2 \frac{2}{1+j} = \frac{24}{3k^4+10k^3+9k^2+2k} \approx \frac{1}{k^4}.
\end{align}
For the $L^2$ norm we observe using \eqref{ext::polpropE} and \eqref{ext::polpropF} that
\begin{align*}
  \tilde{e}(x) &=\sum\limits_{j=1}^k -c_j\int_x^1 p_{j-1}^1(s) \intd s = \sum\limits_{j=1}^k -c_j\int_x^1 \frac{1}{j} \sum_{i=0}^{j-1}(2i+1)p_{i}^0(s) \intd s = \sum\limits_{j=1}^k \frac{c_j}{j} \sum_{i=0}^{j-1}\underbrace{-(2i+1)\int_x^1p_{i}^0(s)}_{= (2i+1)l_{i+1}} \intd s\\
               & = \sum\limits_{j=1}^k \frac{c_j}{j} \sum_{i=0}^{j-1}\left( p_{i+1}^0(x) - p_{i-1}^0(x)\right) = \sum\limits_{j=1}^k \frac{c_j}{j} \left( p_{j}^0(x) + p_{j-1}^0(x)\right),
\end{align*}
and so with the definition of the coefficients \eqref{ext::divi::coeffs} and using an inverse inequality (for example in \cite{Bernardi:1997:SMH} page 253) also
\begin{align*}
\| \tilde{e} \|^2_{0,E} = \sum_{j=1}^k \frac{c_j^2}{j^2} \|  p_{j}^0 + p_{j-1}^0 \|_{0,E}^2 \cle \sum_{j=1}^k \frac{j^2}{k^8} \underbrace{ \|  p_{j}^0  \|_{0,E}^2}_{\cle \frac{2}{2j+1}} \cle \frac{j}{k^8} \sum_{j=1}^k1 \cle \frac{1}{k^6} \quad \text{and} \quad \| \tilde{e}\|^2_{1,E} \cle \frac{1}{k^2}.
\end{align*}
Using a linear transformation $F$ from $E$ to $E_1$ we set $e(x) := \tilde{e}(F^{-1}(x))\frac{x^2}{2}$ to see $e(0) = e(1) = e'(0) = 0$ and $e'(1)=1$, and  
\begin{align} \label{ext:divtheorem:twoestimates}
  \| e \|_{0,E_1}  \cle \frac{1}{k^3} \quad \text{and} \quad \| e \|_{1,E_1} \cle \frac{1}{k}.
\end{align}
Similar as in the proof of theorem \ref{ext:tangh2ext} we now use the real method of interpolation of spaces. As $u=0$ on $\partial E_1$ we have $u \in H^1_0(E_1)$, thus together with $H^{1/2}_{00}(E_1) = [L^2(E_1),H_0^1(E_1)]$ and the definition of the norm on an interpolated space we have with \eqref{ext:divtheorem:twoestimates}
\begin{align*}
 \| e\|_{1/2^*, E_1} \cle \sqrt{ \| e\|_{0, E_1}\| e\|_{1, E_1} }  \cle \frac{1}{k^2},
\end{align*}
so \eqref{ext::divtheorem:B} is proven. It remains the first statement. We start by defining an extension from the edge to the triangle by
\begin{align*}
  \psi(u)(x,y):= \int_0^1 a(s) u(x+sy) \intd s \quad \textrm{with} \quad a(s) = 4-6s,
\end{align*}
and so $u'(1) = \frac{\partial \psi}{\partial x}(1,0)$. Again using the techniques of step 1 of the proof of theorem  \ref{ext:tangh2ext} we easily get $\| \psi \|_{1,\Tref} \cle \| u \|_{1/2,E_1}$.
Next we define the mean value along the line $l_x := \{(x,y): 0 \le x \le 1, y \in [0,1-x] \}$
\begin{align*}
\overline{u}(x,y) := \frac{1}{1-x} \int_x^1 \psi(x,s) \intd s =  \int_0^1 \psi(x,(1-x)s) \intd s.
\end{align*}
Due to $\frac{\partial \overline{u}}{\partial y} = 0$, it follows with  $\frac{\partial \overline{u}}{\partial x} := \overline{u}'$,
\begin{align*}
| \overline{u} |_{1,\Tref}^2 = \int_0^1 \int_0^{1-x} \overline{u}'(x)^2 \intd y \intd x = \int_0^1 (1-x) \overline{u}'(x)^2 \intd x = | \overline{u} |_{1*,E_1}^2
\end{align*}
and so 
\begin{align*}
| \overline{u} |_{1*,E_1} = | \overline{u} |_{1,\Tref} \cle \| \psi \|_{1,\Tref} \cle \| u \|_{1/2, E_1}.
\end{align*}
Using \eqref{ext:divtheorem:approx} and $\tilde{e}'(1)=1 $ we furthermore show that 
$| \overline{u}'(1) | \cle k^2 | \overline{u} |_{1^*,E_1} \cle k^2 \| u \|_{1/2, E_1}$,
and as
\begin{align*}
\overline{u}'(1) =  u'(1) \overbrace{ \int_0^1a(s) \intd s}^{=1} - \frac{1}{2}  u'(1) \overbrace{\int_0^1a(s)s \intd s}^{=0}  = u'(1), 
\end{align*}
we finally get
\begin{align*}
  | u'(1) | \cle  k^2 \| u \|_{1/2, E_1} \cle  k^2 \| u \|_{1/2^*, E_1}
\end{align*}
\end{proof}
\subsection{Proof of theorem \ref{ext:h2ext}} \label{ext::proofmaintheorem}
\begin{proof}In the first step we use theorem \ref{ext:tangh2ext} to find a function $\exttang(\vu)$ with a proper tangential derivation, thus for the difference $\vu_c:= \vu - \nabla \exttang(\vu)$ we have $\vu_c \cdot \tau = 0$ on the boundary $\partial \Tref$. Now let $\vn_i$ be the normal vector on the edge $E_i$ and $u_{\vn_i}:= \vu_c \cdot \vn_i$, so the remaining error in the normal derivation after the first step. The idea is now to split this error in two parts to use theorem \ref{ext:normh2ext} and theorem \ref{ext:dividetheorem}. We start with the lower edge $E_1$ and define $u_{1}:= \vu_{c} \cdot ((x,y) - V_2) \in \pol^{k+1}(\Tref)$, where $V_2 = (0,1)$ is the vertex opposite of $E_1$. As $((x,y) - V_2) \approx \tau$ on the edges $E_2$ and $E_3$ we have $u_{1}|_{E_2} = u_{1}|_{E_3} = 0$. On the lower edge we have $((x,y) - V_2) = (x,-1)$ and as $\vu_c \cdot \tau = 0$, thus the first component of $\vu_c$ is equal to zero, we get $u_{1}|_{E_1} = u_{\vn_1} \in \pol^k(E_1)$. Using theorem \ref{ext:dividetheorem} we find two functions $e_0,e_1 \in \pol^k(E_1)$ with
\begin{align*}
  e'_1(1) =1, e'_1(0) = e_1(0) = e_1(1) = 0 \quad \textrm{and} \quad   e'_0(0) =1, e'_0(1) = e_0(0) = e_0(1) = 0,
\end{align*}
where we mirrored the edge $E$ in theorem \ref{ext:dividetheorem} to find $e_0$. We are now able to split the error to define a {\it good} and a {\it bad} part on the edge $E_1$ by
\begin{align*}
  u_{\vn_1}^b := (u_{1}|_{E_1})'(1) e_1 + (u_{1}|_{E_1})'(0) e_0 \quad \textrm{and} \quad  u_{\vn_1}^g := u_{\vn_1} - u_{\vn_1}^b. 
\end{align*}
The second function $u_{\vn_1}^g$ is {\it good} in the sense of having a zero of order two in the vertices, so $u_{\vn_1}^g \in \pol_{00}^k(E_1)$,  thus we use theorem \ref{ext:normh2ext}. For the other two edges we proceed similarls (see remark \ref{remark::normalext}) to finally define
\begin{align*}
\ext(\vu):= \exttang(\vu) + \extnorm_{E_1}(u_{\vn_1}^g) + \extnorm_{E_2}(u_{\vn_2}^g) + \extnorm_{E_3}(u_{\vn_3}^g).
\end{align*}
Note that due to \eqref{ext:h2extnorm:propB} and \eqref{ext:h2extnorm:propD} the normal derivative of the different corrections do not interfere. As $\extnorm(u_{\vn_i}^g) = 0$ (see \eqref{ext:h2extnorm:propA}) on the boundary $\partial \Tref$ for $i=1,2,3$ the corresponding tangential derivation is also zero thus we have
\begin{align*}
\nabla \ext(\vu) \cdot \tau = \nabla  \exttang(\vu) \cdot \tau +  \nabla \extnorm_{E_1}(u_{\vn_1}^g)\cdot \tau +  \nabla \extnorm_{E_2}(u_{\vn_2}^g)\cdot \tau + \nabla  \extnorm_{E_3}(u_{\vn_3}^g)\cdot \tau = \nabla  \exttang(\vu) \cdot \tau = \vu \cdot \tau,
\end{align*}
so property \eqref{ext:h2ext:propA} is proven. For $\extnorm_{E_1}(u_{\vn_1}^g)$ we get using \eqref{ext:h2extnorm:propC}, \eqref{ext::divtheorem:B} and \eqref{ext::divtheorem:A} as $u_{1}|_{E_1}=0$ on $\partial E_1$
\begin{align*}
  \| \extnorm_{E_1}(u_{\vn_1}^g) \|_{2,\Tref} &\cle \| u_{\vn_1}^g \|_{1/2^*,E_1} = \| u_{\vn_1} - u_{\vn_1}^b \|_{1/2^*,E_1} \\
  &\cle \| u_{\vn_1} \|_{1/2^*,E_1} + |(u_{1}|_{E_1})'(1)| \| e_1\|_{1/2^*,E_1}  + |(u_{1}|_{E_1})'(0)| \| e_0\|_{1/2^*,E_1} \\
                                        &\cle  \| u_{\vn_1} \|_{1/2^*,E_1} + \|u_{1}\|_{1/2^*,E_1}  \frac{k^2}{k^2}  + \|u_{1}\|_{1/2^*,E_1}  \frac{k^2}{k^2} \cle \|u_{1}\|_{1/2^*,E_1}.
\end{align*}
As $u_{1}|_{E_2} = u_{1}|_{E_3} = 0$ we bound $\|u_{1}\|_{1/2^*,E_1}$ by the $H^1$-norm on the triangle, thus we get the estimate $\| \extnorm_{E_1}(u_{\vn_1}^g) \|_{2,\Tref} \cle \|  u_{1} \|_{1,\Tref}  \cle \| \vu_{c} \|_{1,\Tref}$. With the same arguments for the other two normal extensions and inequality \eqref{ext:h2exttang:propC} it follows property \eqref{ext:h2ext:propC},
\begin{align*}
  \| \ext(\vu) \|_{2,\Tref} \le \| \exttang(\vu) \|_{2,\Tref} + 3 \| \vu_{c} \|_{1,\Tref}  \cle \| \vu \|_{1,\Tref} + \| \vu - \nabla \exttang(\vu) \|_{1,\Tref} \cle \| \vu \|_{1,\Tref}. 
\end{align*}
To show \eqref{ext:h2ext:propB} first note that on the boundary $\partial \Tref$ we have
\begin{align*}
  \nabla \ext(\vu) \cdot \vn &= \nabla \exttang(\vu) \cdot \vn + \sum\limits_{i=1}^3 \nabla \extnorm_{E_i}(u_{\vn_i}^g) \cdot n =\nabla \exttang(\vu) \cdot \vn + \overbrace{\sum\limits_{i=1}^3 u_{\vn_i}}^{=\vu_c \cdot n} - \sum\limits_{i=1}^3 u_{\vn_i}^b\\
                             & = \nabla \exttang(\vu) \cdot \vn + \vu \cdot \vn -  \nabla \exttang(\vu) \cdot \vn  - \sum\limits_{i=1}^3 u_{\vn_i}^b=  \vu \cdot \vn  - \sum\limits_{i=1}^3 u_{\vn_i}^b,
\end{align*}
and as $u_{\vn_i}^b|_{E_j} = 0$ for $j \neq i$ it follows $ \| \left( \vu - \nabla \ext(\vu) \right) \cdot \vn \|_{0,\partial \Tref} \le \sum\limits_{i=1}^3 \|  u_{\vn_i}^b \|_{0,E_i} $. As before we use \eqref{ext::divtheorem:B} and \eqref{ext::divtheorem:A} to get
\begin{align*}
  \|  u_{\vn_1}^b \|_{0,E_1} &\cle  |(u_{1}|_{E_1})'(1)| \| e_1  \|_{0,E_1} + |(u_{1}|_{E_1})'(0)| \| e_0 \|_{0,E_1} \cle \frac{1}{k} \|u_1\|_{1/2^*,E_1}\\
                             &\cle \frac{1}{k} \|u_1\|_{1,\Tref} \cle \frac{1}{k} \|\vu_c\|_{1,\Tref} = \frac{1}{k} \|\vu - \nabla \exttang(\vu) \|_{1,\Tref} \cle \frac{1}{k}\|\vu\|_{1,\Tref},
\end{align*}
and with a similar estimate for $u_{\vn_2}^b$ and $u_{\vn_3}^b$ we finally get \eqref{ext:h2ext:propB}
\begin{align*}
\|  \left( \vu - \nabla \ext(\vu) \right) \cdot \vn \|_{0,\partial \Tref} \cle  \frac{1}{k}\|\vu\|_{1,\Tref}.
\end{align*}
\end{proof}

\newpage
\bibliographystyle{IMANUM-BIB}
\bibliography{references}
\clearpage

\end{document}